\documentclass[1p,number]{elsarticle}
\usepackage{amsbsy,amssymb}
\usepackage{graphicx,color}
\graphicspath{{Graph/}}
\usepackage{subeqnarray}
\usepackage{tikz}
\usepackage{comment}
\def\mysep{******************************}
\specialcomment{myc1}{\newline\noindent\mysep\newline\begingroup\ttfamily\footnotesize}%
{\newline\endgroup\noindent\mysep\newline}
\specialcomment{myc2}{\newline\noindent\mysep\newline\begingroup\ttfamily\footnotesize}%
{\newline\endgroup\noindent\mysep\newline}
\specialcomment{myc3}{\newline\noindent\mysep\newline\begingroup\ttfamily\footnotesize}%
{\newline\endgroup\noindent\mysep\newline}
\excludecomment{myc1}
\excludecomment{myc2}
\excludecomment{myc3}
\newtheorem{theorem}{Theorem}

\newdefinition{remark}{Remark}
\newproof{proof}{Proof}
\newproof{proofoft}{Proof of Theorem \ref{thm2}}
\def\d{{\rm d}}

\newcommand{\R}{\mathbb{R}}
\newcommand{\E}{\mathbb{E}}

\newcommand{\Espace}{\mathcal E}
\newcommand{\Jspace}{E}
\newcommand{\Jspacek}{{\Jspace}_k}
\newcommand{\Jspacep}{{\mbox{$E^+$}}}
\newcommand{\Jspacem}{{\mbox{$E^-$}}}
\newcommand{\Jspacepm}{{\mbox{$E^\pm$}}}

\newcommand{\diag}{\mathrm {diag}}

\newcommand{\col}{\mathrm {col}}

\renewcommand{\t}{\top}%{\texttt{t}}

\def\Agen{\mathcal{A}}
\def\Adomain{\mathcal{D}}
\renewcommand{\Pr}{\mathbb{P}}
\let\Min\wedge
\let\Max\vee
\newcommand{\bs}{\boldsymbol}
\newcommand{\matDelta}{\bs\Delta}
\newcommand{\matTheta}{\bs\Theta}

\newcommand{\matI}{{\mathbb I}}
\newcommand{\matO}{{\mathbb O}}

\newcommand{\matA}{{\bs A}}
\newcommand{\matM}{{\bs M}}
\newcommand{\matE}{{\bs E}}
\newcommand{\matC}{{\bs C}}
\newcommand{\matP}{{\bs P}}

\newcommand{\matF}{{\bs F}}
\newcommand{\matG}{{\bs G}}
\newcommand{\matL}{{\bs L}}
\newcommand{\matS}{{\bs S}}
\newcommand{\matQ}{{\bf Q}}

\newcommand{\matX}{\bs X}
\newcommand{\matJ}{\bs J}
\newcommand{\matLambda}{\bs \Lambda}
\newcommand{\matGamma}{\bs \Gamma}

\newcommand{\matD}{{\bs D}}
\newcommand{\matDp}{\matD^+}
\newcommand{\matbD}{{\bar{\bs D}}}
\newcommand{\matbDp}{\matbD^+}
\newcommand{\mattD}{{\tilde{\bs D}}}

\newcommand{\matU}{{\bs U}}
\newcommand{\matUm}{\matU^-}
\newcommand{\matbU}{{\bar{\bs U}}}
\newcommand{\matbUm}{\matbU^-}
\newcommand{\mattU}{{\tilde{\bs U}}}

\newcommand{\vecPi}{\vec \Pi}
\newcommand{\vecpi}{\vec \pi}

\newcommand{\vecb}{\vec b}
\newcommand{\vecc}{\vec c}
\newcommand{\vecu}{\vec u}
\newcommand{\veck}{\vec k}

\def\Ito{It\^{o}}

\begin{document}
% ----------------------------------------------------------------
\title{Two-sided reflected Markov-modulated Brownian motion with applications
to fluid queues and dividend payouts}
%%%%%%%%%%%%%%%%%%%%% AUTHORS %%%%%%%%%%%%%%%%%%%%
\author{Bernardo D'Auria\corref{cor}\fnref{fnD}}
\ead{bernardo.dauria@uc3m.es}
\ead[url]{http://www.est.uc3m.es/bdauria}
\address{Universidad Carlos III de Madrid,%
Avda. Universidad, 30%
28911 Legan\'es (Madrid) Spain.}
%%%%%%%%%%%%%%%%%%%%%%%%%%%%%%%%%%%%%%%%%%%%%%%%%%%%%%%%%%%%%%%%%
\author{Offer Kella\fnref{fnK}}
\ead{Offer.Kella@huji.ac.il}
\ead[url]{http://pluto.mscc.huji.ac.il/~mskella/kella.html}
\address{Department of Statistics,%
The Hebrew University of Jerusalem,% 
Jerusalem 91905, Israel.}
%%%%%%%%%%%%%%%%%%% FOOTNOTES %%%%%%%%%%%%%%%%%%%%%%%%%%%%%
\cortext[cor]{Corresponding author}
\fntext[fnD]{This research has been partially supported by the Spanish Ministry of
Education and Science Grants MTM2007-63140, MTM2010-16519, SEJ2007-64500 and RYC-2009-04671.}
\fntext[fnK]{Supported in part by grant 434/09 from the Israel Science Foundation and the Vigevani Chair in Statistics}
%%%%%%%%%%%%%%%%%%%%%%%%%%%%%%%%%%%%%%%%%%%%%%%%%%%%%%%%%%%%%%%%%
%************************************************************************
%			Keywords and Codes
%************************************************************************
%\MSCcodes{Primary:   60K25; Secondary:  60K37 }
% {\em See} \url{http://www.ams.org/msc}      }
% See the MSC2000 codes at http://www.ams.org/msc/
%\ORMScodes{Primary:  Queues,  Probability; Secondary:   Markov Processes, Diffusion Models }
% {\em See} \url{http://or.pubs.informs.org/Media/ORSubject.pdf}  }
% See the OR/MS classification at http://or.pubs.informs.org/Media/ORSubject.pdf
%************************************************************************
%			Abstract and Introduction				
%************************************************************************
\begin{abstract}
In this paper we study a reflected Markov-modulated Brownian motion with a two sided reflection in which the drift, diffusion coefficient and the two boundaries are (jointly) modulated by a finite state space irreducible continuous time Markov chain. The goal is to compute the stationary distribution of this Markov process, which in addition to the complication of having a stochastic boundary can also include jumps at state change epochs of the underlying Markov chain because of the boundary changes. We give the general theory and then specialize to the case where the underlying Markov chain has two states. Moreover, motivated by an application of optimal dividend strategies, we consider the case where the lower barrier is zero and the upper barrier is subject to control. In this case we generalized earlier results from the
case of a reflected Brownian motion to the Markov modulated case.
\end{abstract}
\maketitle
% ----------------------------------------------------------------
\section{Introduction}
A double sided reflected process, say at zero from below and some positive level $b$, is a
reasonable model for a storage process where the stored quantity has to be nonnegative and the
buffer size is limited. When borrowing or backlogging is allowed, then the lower barrier could also be negative. There is a huge literature on such processes, in particular when the driving process (before reflection) is Brownian motion. Less attention is given to the case where the boundaries are themselves stochastic processes. For most papers on this topic the focus was on showing the existence and uniqueness of solution of the related Skorohod problem. A recent study which refers to many of the earlier results in this particular direction is \cite{nystrom:onkog:2010} where the focus is on multidimensional models. 
For the one dimensional double sided reflection (non-modulated case), we mention the important
results reported in \cite{kruk:lehoczky:ramanan:shreve:2007} and references therein.

Very little work is done related to the computation of the stationary distribution of such processes when more explicit stochastic structure is assumed, especially when the boundaries are not smooth.
One example of such a study is given in \cite{kella:boxma:mandjes:2006} where the driving process is L\'{e}vy and there is only one lower boundary which increases linearly and then drops back to zero at arrival epochs of a Poisson process.

We are not aware of any results for the case where the boundary together with the driving process
are jointly modulated by some other process. This is motivated by situations in which the buffer
size and the allowed backlog are allowed to change from time to time as a response to changes in the
driving process which are caused by changes in an underlying environment.

In this paper we model the environment as a finite state space irreducible continuous time Markov
chain. When in a given state, our process behaves like a two sided reflected Brownian motion with
drift and diffusion coefficient as well as lower and upper boundaries which are allowed to depend on
this state. The main goal is to give a computational scheme for computing the joint stationary
distribution of the buffer content and the state of the underlying environment.

The paper is organized as follows. 
In Section~\ref{sec.model} we present the general model and provide some
preliminary results. 
Section~\ref{sec.stat-dist} is about the stationary joint distribution of the buffer content and the underlying environment. 
Section~\ref{sec:fluid-queues} specialized the results to various cases where the lower
barrier is zero (no backlog) and the underlying environment changes between two states. Under some conditions, for this case we also show how  to compute the distribution of some regenerative epoch associated with this process. 
Finally, in Section~\ref{sec:dividend-payout} we generalize results of \cite{gerber:shiu:2004}, who considered the upper barrier as a cutoff point above which a company
must pay dividends that are modeled by the regulating process at this upper barrier.
%************************************************************************
%			Other sections
%************************************************************************
%%%%%%%%%%%%%%%%%%%%%%%%%%%%%%%%%%%%%%%%%%%
%%%%%%%%%%%%%%%%%%%%%%%%%%%%%%%%%%%%%%%%%%%
    %************************************************************************
%			Model				
%************************************************************************
\section{Model}\label{sec.model}
Let $W=\{W(t)|\ t\ge 0\}$ and $J=\{J(t)|\ t\ge 0\}$ be two independent processes where $W$ is a
Wiener process (a standard Brownian motion) and $J$ is an irreducible and homogeneous
continuous time Markov chain with state space $\Jspace = \{1,\ldots,N\}$. We assume that $J$ has
right continuous sample paths and we denote by $\matQ=(q_{ij})$ its rate transition matrix, by
$\vecpi=(\pi_i)$ its stationary distribution and define $\matP=\diag[\vecpi]$.
For each $i\in\Jspace$ we let $a(i) \le b(i)$ be two finite real numbers which define the upper
and lower barriers when in state $i$.

It is standard to show that there is a unique process $(Z,L,U)$ satisfying
\begin{equation}\label{eq:(Z,J).def}
Z(t)=Z(0)+\int_0^t \sigma(J(s)) \, dW(s) + \int_0^t \mu(J(s)) \, ds + L(t) - U(t)
\end{equation}
where $a(J(t))\le Z(t)\le b(J(t))$ for each $t\ge 0$,
$L$ and $U$ are nondecreasing right continuous processes with $L(0-)=U(0-)=0$,
\begin{equation}\label{eq:L.U}
\int_0^\infty \Big( Z(s) - a( J(s) ) \Big) \, dL(s)=0 
\quad \mbox{ and } \quad
\int_0^\infty \Big( b( J(s) ) - Z(s) \Big) \, dU(s)=0\ .
\end{equation}
$Z$ is the two-sided (Skorohod) reflection of the modulated process 
\begin{equation}\label{eq:X.proc}
X(t) = Z(0) + \int_0^t \sigma(J(s)) \, dW(s) + \int_0^t \mu(J(s)) \, ds 
\end{equation}
at the modulated barriers $[a(i),b(i)]$, $i\in\Jspace$. We denote by $\kappa = \sum_i \mu(i) \pi_i$
the asymptotic drift of the process $X(t)$. 

Although $Z$ is not Markovian, $(Z,J)$ is. Let us identify its generator.

Let $f(w,i)$ be a bounded twice continuously differentiable function in $w$ satisfying 
$f'(a(i),i)=f'(b(i),i)=0$.
We note that clearly there always is a twice continuously differentiable 
$h:\R^2 \to \R$ such that $h(w,i)=f(w,i)$. 
The generalized It\^o formula for semimartingales 
(e.g. Theorem 33 on p. 8 of \cite{protter:2004}) 
now implies after some obvious manipulations that
\begin{eqnarray}\label{eq:Ito.formula}
f(Z(t),J(t))&=& f(Z(0),J(0)) + \int_0^t\sigma(J(s))f'(Z(s),J(s)) \, dW(s)\\
&&+\int_0^t \frac{\sigma^2(J(s))}{2}f''(Z(s),J(s))+\mu(J(s))f'(Z(s),J(s)) \, ds \nonumber \\
&&+\int_0^t f'(Z(s),J(s)) \, d(L^c(s)-U^c(s)) + \sum_{0<s\le t}\Delta f(Z(s),J(s))\nonumber
\end{eqnarray}
where for c\`adl\`ag functions of bounded variation on compact sets $g$ we denote
$g(s-)=\lim_{u\uparrow s}g(u)$, $\Delta
g(s)=g(s)-g(s-)$ and $g^c(s)=g(s)-\sum_{0<u\le s}\Delta g(u)$.
Now, observe that  $f'(a(i),i)=0$ and the first relation in (\ref{eq:L.U}) imply that
\begin{equation}
\int_0^t f'(Z(s),J(s)) \, dL^c(s)=\int_0^t f'(a(J(s)),J(s)) \, dL^c(s)=0 \nonumber
\end{equation}
and similarly $\int_0^t f'(Z(s),J(s)) \, dU^c(s)=0$.

Next, note that the only way that $s$ can be a jump epoch of $f(Z(\cdot),J(\cdot))$ is if it is a
jump epoch of $J$. If $a(J(s))\le Z(s-)\le b(J(s))$ then clearly $Z(s)=Z(s-)$. If either
$Z(s-)<a(J(s))$ or $Z(s-)>b(J(s))$, then in the first case necessarily $Z(s)=a(J(s))$ and in the
second $Z(s)=b(J(s))$, therefore we can always write $Z(s)= a(J(s)) \vee Z(s-) \wedge b(J(s)) $.
This implies that
\begin{equation}
\Delta f(Z(s),J(s))
=\hat f(Z(s-),J(s)) - \hat f(Z(s-),J(s-))\ , \nonumber
\end{equation}
where we used the notation $\hat f(w,i)= f(a(i) \Max w \Min b(i),i)$.
Now for each $i\not=j$ let $N_{ij}$ be independent Poisson process with rate $q_{ij}$, independent
of $W$, such that if an arrival finds $J$ in state $i$ then it instructs $J$ to jump to $j$ and
otherwise nothing happens. Now, $N_{ij}(s)-q_{ij}s$ is a martingale with respect to the
filtration generated by $W$ and $N_{ij}$ for $i\not=j$ and $\hat f(Z(s),i)$ are bounded processes. Hence, 
\begin{equation}
%\int_{[0,t]} \big( f(a(j) \vee Z(s-) \wedge b(j),j)-f(Z(s-),i) \big)
%    1_{\{J(s-)=i\}}  \, d(N_{ij}(s)-q_{ij}s)
\int_{[0,t]} \big( \hat f(Z(s-),j) - \hat f(Z(s-),i) \big)
    1_{\{J(s-)=i\}}  \, d(N_{ij}(s)-q_{ij}s) \nonumber
\end{equation}
is a martingale and thus, by summing over $i,j\in\Jspace$, with $i\not=j$, and noting that
$q_{ii}=-\sum_{j\not=i}q_{ij}$, this implies that
\begin{equation}
\sum_{0<s\le t}\Delta f(Z(s),J(s))-\int_0^t \sum_{j \in \Jspace} q_{J(s-),j}
%f(a(j) \vee Z(s-) \wedge b(j),j)
\hat f(Z(s-),j) \, ds \nonumber
\end{equation}
is a martingale with respect to the filtration generated by $W$ and the Poisson processes $N_{ij}$,
but also with respect to the filtration generated by $W$ and $J$, since it is adapted to that
filtration and $N_{ij}$ are non-anticipative with respect to it 
(see also \cite{rogers:williams:2000},  Lemma V.21.13).
\begin{myc1}
It could be useful to define $N_{ii} \equiv 0$, with $i\in\Jspace$, in order to avoid the condition
$i\not=j$ when summing over both indexes. The process $J$ may be expressed in terms of the
underlying Poisson processes in the following way
\begin{eqnarray*}
J(t) - J(0) = \sum_{0<s\le t} \Delta J(s)
&=& \sum_{i,j\in\Jspace} \int_{[0,t]} (j-i) 1_{\{J(s-)=i\}} \, dN_{ij}(s)\\
&=& \sum_{i,j\in\Jspace} \int_{[0,t]} (J(s)-J(s-)) 1_{\{J(s-)=i\}} \, dN_{ij}(s)
\end{eqnarray*}
and similarly any function of it
\begin{eqnarray*}
f(J(t)) - f(J(0)) = \sum_{0<s\le t} \Delta f(J(s))
&=& \sum_{i,j\in\Jspace} \int_{[0,t]} (f(j)-f(i)) 1_{\{J(s-)=i\}} \, dN_{ij}(s)\\
&=& \sum_{i,j\in\Jspace} \int_{[0,t]} (f(J(s))-f(J(s-)) 1_{\{J(s-)=i\}} \, dN_{ij}(s),
\end{eqnarray*}
where $f$ is any function defined over $\Jspace$.
From this it follows easily that
\begin{eqnarray*}
\sum_{0<s\le t}\Delta f(Z(s),J(s))
&=& \sum_{i,j\in\Jspace}
\int_{[0,t]}\Big(f(Z(s),J(s))-f(Z(s-),J(s-))\Big)1_{\{J(s-)=i\}} \, dN_{ij}(s)\\
&=& \sum_{i,j\in\Jspace}
\int_{[0,t]}\Big(f(a(J(s))\vee Z(s-)\wedge b(J(s)),J(s))-f(Z(s-),J(s-))\Big)
    1_{\{J(s-)=i\}} \, dN_{ij}(s)\\
&=& \sum_{i,j\in\Jspace}
\int_{[0,t]}\Big(f(a(j)\vee Z(s-)\wedge b(j),j)-f(Z(s-),i)\Big)1_{\{J(s-)=i\}} \, dN_{ij}(s)
\end{eqnarray*}
In addition notice that
\begin{eqnarray*}
\lefteqn{\sum_{\stackrel{i,j\in\Jspace}{i\not=j}} \int_{[0,t]}\Big(f(a(J(s))\vee Z(s-)\wedge
b(J(s)),J(s))-f(Z(s-),J(s-))\Big)
    1_{\{J(s-)=i\}} \,  q_{ij} \, ds }\\
&=&\sum_{\stackrel{i,j\in\Jspace}{i\not=j}} \int_{[0,t]}\Big(f(Z(s),J(s))-f(Z(s-),J(s-))\Big)
    1_{\{J(s-)=i\}} \,  q_{J(s-),j} \, ds \\
&=&\sum_{j\in\Jspace}\int_0^t f(Z(s),j) \sum_{i\not=j} 1_{\{J(s-)=i\}} \,  q_{J(s-),j} \, ds
 -\sum_{i\in\Jspace}\int_0^t f(Z(s-),i) 1_{\{J(s-)=i\}} \,  \sum_{j\not= i} q_{J(s-),j} \, ds \\
&=&\sum_{j\in\Jspace}\int_0^t f(Z(s),j) 1_{\{J(s-)\not=j\}} \,  q_{J(s-),j} \, ds
 +\sum_{i\in\Jspace}\int_0^t f(Z(s-),i) 1_{\{J(s-)=i\}} \,  q_{J(s-),i} \, ds \\
&=&\sum_{j\in\Jspace}\int_0^t f(Z(s),j) 1_{\{J(s-)\not=j\}} \,  q_{J(s-),j} \, ds
 +\sum_{j\in\Jspace}\int_0^t f(Z(s-),j) 1_{\{J(s-)=j\}} \,  q_{J(s-),j} \, ds\\
&=&\sum_{j\in\Jspace}\int_0^t f(Z(s),j)  \,  q_{J(s-),j} \, ds.
\end{eqnarray*}
\end{myc1}
Since $f'(w,i)$ are locally bounded, it follows that the following is a martingale with
respect to the filtration generated by $W$ and $J$:
\begin{eqnarray*}
M(t) &=& \sum_{0<s\le t}\Delta f(Z(s),J(s)) -\int_0^t \sum_{j \in \Jspace} q_{J(s),j}
\hat f(Z(s),j) \, ds \\ && +\int_0^t\sigma(J(s))f'(Z(s),J(s)) \, dW(s)\ .
\end{eqnarray*}
Rewriting equation (\ref{eq:Ito.formula}) in terms of the above martingale we have
\begin{eqnarray}\label{eq:Stroock-Varadhan}
f(Z(t),J(t))-f(Z(0),J(0))&=&\int_0^t \Agen f(Z(s),J(s)) \, ds + M(t) \nonumber
\end{eqnarray}
where we denoted by $\Agen f(z,i)$ the following operator
\begin{equation}\label{eq:generator}
 \Agen f(z,i)=\frac{1}{2}\sigma^2(i) f''(z,i)+\mu(i) f'(z,i)
     +\sum_{j \in \Jspace} q_{ij} \hat f(z,j), \quad (z,i)\in\Espace \ .
\end{equation}

This gives the expression for the generator of $(Z,J)$ restricted
to the set of bounded functions twice continuously differentiable on the continuous component and
satisfying $f'(a(i),i)=f'(b(i),i)=0$, $i\in\Jspace$.

We denote by $\Espace=\bigcup_{i\in\Jspace} [a(i),b(i)]\times \{i\}$
the range of values assumed by the process $(Z,J)$ and define the following subsets of the state
space of $J$,
$$
    \Jspacep = \{j \in E; \sigma(j) >0 \mbox{ or } \mu(j) > 0\} \ , 
\quad \mbox{ and } \quad 
    \Jspacem = \{j \in E; \sigma(j) >0 \mbox{ or } \mu(j) < 0\} \ .
$$
%We also define $\Npm:=|\Jspacepm|$.
\begin{myc1}
Denoting by
$$\Adomain=\{f \in C_b^1(\R\times\Jspace): f'(a(i),i)=f'(b(i),i)=0, i\in\Jspace\}$$
we see that $\Adomain \subset \Adomain(\Agen)$.
Actually it is a proper subset of the domain of the generator and alto it does not form a core for
$\Agen$.
Indeed if we take $f\in\Adomain$ we have that $\Agen f \not \in \Adomain$.

However given $\Agen$ as in (\ref{eq:generator}) defined on $\Adomain$, it satisfies the positive
maximum principle, \cite[see][pag.165]{ethier:kurtz:1986}, and having that $\Adomain$ is dense in
$C_b^0(\R\times\Jspace)$, i.e. the set of all bounded continuous functions on
\hbox{$\R\times\Jspace$,}  by Theorem IV.5.4 pag. 199 in \cite{ethier:kurtz:1986} there exists a
solution of the martingale problem for $\Agen$.
In addition since $\Agen 1\{\Espace\} = 0$, by Remark IV.5.5 the solution lives on
$D_{\Espace}[0,\infty)$.

Actually one possible solution is the process $(Z, J)$ defined in (\ref{eq:(Z,J).def}). It follows
by noticing that equation (\ref{eq:Stroock-Varadhan}) is equal to formula (1.30) in Proposition
IV.1.7 page 162 in \cite{ethier:kurtz:1986}.
However if we extend equation (\ref{eq:generator}) to any $(z,i)\in\R\times\Jspace$, then we get
that the solution of the martingale problem is not unique and does not live only on $\Espace$.
There are at least two possible Markov processes that solve the martingale problem for $\Agen$.
One solution is $(Z,J)$, i.e. assume that the process starts in $(Z(0),J(0)$ with $Z(0)\not \in
[a(J(0)),b(J(0))]$. Then the process moves till hitting the boundary or till a change of the state
$J(0)$ and then remains confined in $\Espace$. The other solution is the opposite one.
It starts in $\Espace$ and moves to the boundary where it is obligated to live outside, at least
till a change of the state of the environment $J$.

To make the story easier consider the standard Brownian motion on the line reflected at the origin.
The generator is $\frac{1}{2} f''(x)$ for any function that has $f'(0)=0$.

If we consider $\R$ as the range of the process, the generator defined as above is not enough to
specify the process.
Indeed we could have that the reflection is only allowed to live either on $\R^+$ or $\R^-$.
For example, consider the second case, starting in \hbox{$x\in\R^+$.} the process moves as the
Brownian
motion till hitting the level $0$, i.e. $T_0$, and then continues as $-|B(t-T_0)|$.
What distinguishes the two processes are the domains of the corresponding generators.
The domain of the generator of the process that eventually lives on $\R+$ contains all the functions
that have null right derivative in $0$, while the generator of the process that eventually
lives on $\R^-$ contains in the domain all the functions that have null left derivative in $0$.

The question now is, given the following operator
\begin{equation}\label{eq:extended.generator}
\Agen f(z,i)=\frac{1}{2}\sigma^2(i) \hat f''(z,i)+\mu(i) \hat f'(z,i)
    +\sum_{j \in \Jspace} q_{ij} \hat f(z,j), \quad (z,i)\in\R\times\Jspace
\end{equation}
and noticing that the process jumps immediately in $\Espace$, due to the \^{} operator applied to
the first and second derivatives of $f$, there is no doubt about where the process lives.
Therefore, is the restriction $\Agen_{|\Adomain}$ enough to uniquely determine the process
$(Z,J)$? Or in other words, is the closure of $\Agen_{|\Adomain}$ as a linear operator over
$C_b^0(\R\times\Jspace)$ (we could address the same question using the definition
(\ref{eq:generator}) over $C_b^0(\Espace)$) unique? [\emph{I'm not going to answer this
question in this paper}]
\end{myc1}
%		Approximation Theorems 		
%************************************************************************
%\input{Approximation_via_RWs}
%************************************************************************
%		System of Differential Equations		
%************************************************************************
\section{Stationary Distribution}\label{sec.stat-dist}
If $i\in\Jspacem$, $a(i) \leq Z(0) \leq b(i)$ and $J(0)=i$, then the probability of hitting $a(i)$ before $J$ changes state is bounded below by the positive probability of this event starting from $Z(0) = b(i)$ and $J(0) = i$. This observation together with a geometric retrial argument, recalling that $J$ is irreducible, implies that $(Z,J)$ is a regenerative process with finite mean regeneration epochs. A similar argument can be made for $i\in\Jspacep$. 
If $\Jspacem \cup \Jspacep = \emptyset$ then the process is just a deterministic function of $J$ and thus clearly positive recurrent. Thus, in any case a unique stationary distribution exists. 

We show that if a solution to (\ref{eq:sys.diff.eq}) exists then it must be the (unique) stationary distribution. Later we will show how to construct it. 
%To avoid trivial cases we assume that the set $\Jspacep \cup \Jspacem$ is not empty and 
%without any loss of generality we assume that 
%we can choose $i\in\Jspacem$ such that $a(i) < b(i)$. 
%Consider the \emph{\hbox{first-relevant}} returning time to $0$, 
%i.e. $\tau=\inf_{t \geq T_1}\{Z(t)=0,\ J(t)=i\}$
%under the probability $\Pr_{(0,i)}$ that the process starts at $(Z(0)=0,J(0)=i)$.
 %Consider an inter-jump epoch $[T_n,T_{n+1})$, $n\geq2$, during which the state of $J$ is fixed to
%$i$, the probability that $Z$ hits $0$ during this interval is at least the probability that a
%$(\mu(i),\sigma(i))-$Brownian motion starting at $b_i$ hits $0$ before time $T_{n+1}-T_n$, which is
%strictly positive. This, together with a standard geometric retrial argument implies that, since $J$
%is irreducible, the states $(0,i)$, $i\in\Jspacem$ are all regeneration states for $(Z,J)$ with
%finite
%non-zero mean regeneration epochs clearly having a density (i.e., non-arithmetic). Thus a unique
%stationary distribution for $(Z,J)$ exists and denoting by $(Z^*,J^*)$ a pair of random variables
%distributed according to it we define $\Pi_i(z)=\Pr\{Z^* \leq z, J^* = i\}$.
%
%The following theorem provides a way to compute the stationary distribution.
\begin{theorem}\label{thm:stat.dist}
 The stationary distribution of the process $(Z,J)$ is the unique solution of the following system
of differential equations
\begin{equation}\label{eq:sys.diff.eq}
  \frac{1}{2} \sigma^2(i) \, \Pi''_i(z) -\mu(i) \, \Pi'_i(z)
  + \sum_{j\in\Jspace} q_{ji} \, \Pi_j(a(j) \Max z \Min b(j)) = 0
    \quad a(i) \leq z \leq b(i)
\end{equation}
with boundary conditions 
$\Pi_i(a(i))= 0$ and $\Pi_i(b(i))=\pi_i$, $i\in\Jspace$.
\end{theorem}
\begin{proof}
The stationary distribution satisfies the following equation for any function $f$ belonging to the
domain of the generator $\Agen$,
\begin{equation}\label{eq:inv.distr.and.genA}
\sum_{i\in\Jspace} \int_{a(i)}^{b(i)} \Agen f(z,i) \, \d\Pi_i(z) = 0 \ ,
\end{equation}
Let $f$ be any twice continuous differentiable function on $\Espace$ with bounded support,
then using integration by parts we get that equation (\ref{eq:inv.distr.and.genA}) 
reduces to 
\begin{equation}\label{eq:int.eq}
\sum_{i\in\Jspace} \int_{a(i)}^{b(i)} \frac{1}{2} \sigma^2(i) \ \Pi''_i(z) - \mu(i)\ \Pi'_i(z) 
    \ \d f(z,i)
    + \sum_{i\in\Jspace} \int_{a(i)}^{b(i)} \sum_{j\in\Jspace} q_{ij} \ \Pi_i(z)  
    \ \d \hat f(z,j) = 0 \ .
\end{equation}
Noting that 
$$\int_{a(i)}^{b(i)} \Pi_i(z)  \ \d \hat f(z,j) = \int_{a(j)}^{b(j)} \Pi_i(a(i) \Max z \Min b(i))  
    \ \d f(z,j)$$
and interchanging the indexes $i$ and $j$ in the two sums in the last term of
(\ref{eq:int.eq}) we get that (\ref{eq:sys.diff.eq}) implies (\ref{eq:inv.distr.and.genA}). 
% Finally we trivially have $\Pi_i(a(i))= 0$ and $\Pi_i(b(i))=\Pr\{J^*=i\}=\pi_i$.
\end{proof}

The way of solving the system (\ref{eq:sys.diff.eq}) is to divide the interval
$[\min_i\{a(i)\},\max_i\{b(i)\}]$ to disjoint subintervals, 
to get a solution of the above system in each of them and then 
to appropriately glue together all these partial solutions.
For this we set $l_0=\min_i\{a(i)\}$ and by 
$l_{k+1}=\min_i\{a(i)>l_k\} \Min \min_i\{b(i)>l_k\}$ 
and we define the closed intervals $I_k=[l_{k-1},l_k]$, $k=1,\ldots,K$.

Fix one of these subintervals, say $I_k$, and let 
$\Jspacek=\{i\in\Jspace: a(i) \le l_{k-1} < l_k \le b(i)]\}$ be the set of states active over $I_k$.
The restriction of the system (\ref{eq:sys.diff.eq}) to the subinterval $I_k$ then reads as follows
\begin{equation}\label{eq:sys.diff.eq.Ik}
  \frac{1}{2} \sigma^2(i) \, \Pi''_i(z) -\mu(i) \, \Pi'_i(z)
  + \sum_{j\in\Jspacek} q_{ji} \, \Pi_j(z) = c_k(i)
    \quad l_{k-1} \le z \le l_k, \ i\in\Jspacek
\end{equation}
where we set $c_k(i)=\sum_{j:b(j)\le l_k} q_{ji} \, \pi_j$.
The equivalent matrix form of (\ref{eq:sys.diff.eq.Ik}) is the following
\begin{equation}\label{eq:vec.sys.diff.eq.Ik}
\matS_k \vecPi_k''(z) -  \matM_k \vecPi_k'(z) + \matQ_k^\t \vecPi_k (z) = \vecc_k
\end{equation}
where 
$\matS_k = \diag[i\in\Jspacek:\sigma^2(i)/2]$, 
$\matM_k = \diag[i\in\Jspacek:\mu(i)]$, 
$\matQ_k = [(i,j\in\Jspacek:q_{ij})]$ and where 
$\vecPi_k(z)$ and $\vecc_k$ denote the restrictions of the vectors $\vecPi(z)$ and $\vecc$
to the only states active over $I_k$.
$[\cdot]^\t$ is used to denote the transposition operator.

If the set of active states over $I_k$ is a proper subset of $\Jspace$ we have that $\matQ$ is
a strictly substochastic matrix and this has an inverse, therefore the system
(\ref{eq:vec.sys.diff.eq.Ik}) admits the constant 
$\veck_k = [\matQ^{-1}]^\t \ \vecc$ 
as particular solution.
In the case $\Jspacek = \Jspace$, $\matQ$ reduces to the rate transition matrix of
$J$ that is stochastic and singular. For this case the constant $\vecc$ is zero,
the system (\ref{eq:vec.sys.diff.eq.Ik}) is homogeneous and the particular solution 
$\veck_k = \vec0$
is the zero constant. 

Adding to the particular solution the homogeneous solution, we can always write 
the general solution of (\ref{eq:vec.sys.diff.eq.Ik}) in the following form,
see \cite{gohberg:lancaster:rodman:1982},
\begin{equation}\label{eq:Jord.Pair}
\vecPi_k(z) = \matGamma_k e^{\matLambda_k z}  \ \vecu_k + \veck_k. 
\end{equation}
where $(\matGamma_k,\matLambda_k)$ is a Jordan pair of the matrix polynomial
(\ref{eq:vec.sys.diff.eq.Ik}) and $\vecu_k$ is the unknown vector that can be determined 
by the boundary conditions.
For the specific case of second order matrix polynomial, 
a Jordan pair consists of a pair of matrix with the following 
properties
$$ 
\matS_k \, \matGamma_k \, \matLambda_k^2 
    + \matM_k \, \matGamma_k \, \matLambda_k
	+ \matQ_k \, \matGamma_k = \matO $$
and the rank of $\col[\matGamma_k, \matGamma_k \, \matLambda_k]$ is maximum.

To be able to glue together the solutions of all intervals $I_k$, we will see later
that it is necessary to be able to solve a system of equations, and this requires 
to prove that there exists the inverse of the matrix associated to the system.

Let $\matP_k$ be the restriction of the matrix $\matP$ to the states active on $I_k$,
the main ingredient follows from a  slight generalization of the results in
\cite{dauria:ivanovs:kella:mandjes:2010} where it is shown that the systems
\begin{equation}\label{eq:vec.sys.eq.Ik}
\matS_k \, \matG \, \matL^2  
    \mp  \matM_k \, \matG \, \matL
	+ \matP_k^{-1} \matQ_k^\t \matP_k \, \matG  =
\matO 
\end{equation}
admit solutions $(\matGamma^\pm_k,\matLambda^\pm_k)$
that are unique under the restriction that the matrices $\matLambda^\pm$ are substochastic. 
In this case the matrices $\matGamma^\pm$ are transition matrices and in particular 
$\matE^\pm \ \matGamma^\pm = \matI^\pm$ where the projection matrices $\matE^\pm$ are defined as the
submatrices of the identity matrix $\matI_k$ constructed by selecting only the rows corresponding to
the states contained in the set $\Jspacek \cap \Jspacepm$. The substochastic nature of the
matrices  $\matLambda^\pm$ is the special characteristic that we are going to
use later to prove the non singularity in the final system of equation.

The solutions $(\matGamma^\pm_k,\matLambda^\pm_k)$ allow to immediately construct the Jordan
Pair in (\ref{eq:Jord.Pair}) as follows
\begin{eqnarray}\label{eq:Jordan.Pair.special}
\matGamma_k = [\matP_k\matGamma^+_k,\matP_k\matGamma^-_k]; && 
\matLambda_k = \diag[\matLambda^+_k,-\matLambda^-_k] \ .
\end{eqnarray}
This construction is always valid but in the case when $\matLambda^\pm_k$ have 
in common the null eigenvalue. This happens only in the case $\Jspace_k = \Jspace$ and the 
asymptotic drift of the modulated process $X(t)$ is zero, that is $\kappa = 0$, see also
Section~7 in \cite{dauria:ivanovs:kella:mandjes:2010}. 
We are going to exclude this case as it has no added difficulty 
if not the one of making more complex all formulas.

Using the special selected Jordan pair in (\ref{eq:Jordan.Pair.special}),
the solution (\ref{eq:Jord.Pair}) over the interval $I_k$ can be rewritten in the
following form
\begin{equation}\label{eq:interval.sol}
\vecPi_k(z) = 
  \matP_k \, \matGamma_k^+ e^{+\matLambda_k^+ z}  \ \vecu_k^{\ +} 
+ \matP_k \, \matGamma_k^- e^{-\matLambda_k^- z} \ \vecu_k^{\ -}
+ \veck_k \ ,
\end{equation}
and it is fully specified after assigning the unknown boundary values 
\begin{equation}\label{eq:bound.values}
\Pi_i(l_{k-1}+) \mbox{ for any } i \in \Jspacek \cap \Jspacep
\quad \mbox{ and } \quad 
\Pi_i(l_{k}-) \mbox{ for any } i \in \Jspacek \cap \Jspacem \ .
\end{equation} 

To glue together the solution over the entire interval 
$[l_0,l_K]$
it is then necessary to solve for the unknown boundary conditions in
(\ref{eq:bound.values}) by using the constraints
$\Pi_i(a(i))= 0$ and $\Pi_i(b(i))=\pi_i$, for any $i\in\Jspace$,
together with the additional conditions on the continuity of the distribution functions
\begin{subeqnarray}
 \Pi_i(a(j)-) &=& \Pi_i(a(j)+) \quad a(i)<a(j)<b(i)\\
 \Pi_i(b(j)-) &=& \Pi_i(b(j)+) \quad a(i)<b(j)<b(i)
\end{subeqnarray} 
for any $i\in\Jspacep \cup \Jspacem$ and $j\in\Jspace$, 
and on the differentiability of the distribution functions
\begin{subeqnarray} 
 \Pi'_i(a(j)-) &=& \Pi'_i(a(j)+) \quad 	a(i)<a(j)<b(i)\\
 \Pi'_i(b(j)-) &=& \Pi'_i(b(j)+) \quad  a(i)<b(j)<b(i)
\end{subeqnarray} 
for any $i\in\Jspacep \cap \Jspacem$ and $j\in\Jspace$.
%************************************************************************
%  Solving the system of differential equations in the general case
%************************************************************************
\subsection{Computing the stationary distribution}
In this section we show how to determine all the unknowns required to get the unique stationary
distribution.
We are going to consider the sequence of subintervals $I_k$, with $k = 1,\ldots,K$, and we 
represent the solution $\vecPi_k(z)$ in the $k$-th interval according 
to the form (\ref{eq:Jord.Pair}). It is worth to say that when applying the model to specific
cases it is possible to compute numerically the Jordan Pair in (\ref{eq:Jord.Pair}) or in its
special form (\ref{eq:Jordan.Pair.special}), like it is shown for example in \cite{rogers:shi:1994}
or in \cite{asmussen:1995}. We used Wolfram Mathematica\copyright,
to get some of the analytical solutions shown in the examples of Sections~\ref{sec:fluid-queues}
and~\ref{sec:dividend-payout}, but for numerical
computation any technical computing system, such as  Matlab\copyright, can be efficiently used
for this purpose.

In the same way as we defined the projection matrices $\matE_k^\pm$ we define the 
projection matrices for the states that belong to the intersection of the intervals $I_{k-1}$ and
$I_k$, that is the matrices $\matD_k$, $\matDp_k$ and $\mattD_k$. 
They are defined as the submatrices of the identity matrix $\matI_k$ constructed by selecting only
the rows corresponding to the states contained respectively in the sets 
$\Jspace_k \cap \Jspace_{k-1}$, $(\Jspace_k \cap \Jspace_{k-1}) \cap \Jspacep$ and
$(\Jspace_k \cap \Jspace_{k-1}) \cap (\Jspacep \cap \Jspacem)$.
For the states in $\Jspace_k$ that do not belong to the intersection $\Jspace_{k-1} \cap \Jspace_k$
we define the additional projection matrices $\matbD_k$ and $\matbDp_k$ defined as the submatrices
of the identity matrix $\matI_k$ constructed by selecting only the rows corresponding to the states
contained respectively in the sets 
$\Jspace_k \cap \bar \Jspace_{k-1}$, 
$(\Jspace_k \cap \bar \Jspace_{k-1}) \cap \Jspacep$.

Finally considering the intersection between the intervals $I_k$ and $I_{k+1}$ we define the 
corresponding projection matrices $\matU_k$, $\matUm_k$, $\mattU_k$, $\matbU_k$ and $\matbUm_k$
whose definitions are easy to guess. In all definitions we have assumed that
$\Jspace_{0}=\Jspace_{K+1}=\emptyset$.

Applying the continuity and differentiability constraints we get
$$\begin{array}{lll}
 \matbDp_k \ \matGamma_k e^{l_{k-1}\matLambda_k}\  \vecu_k &
    &= \matbDp_k \ \veck_k \\
 \matbUm_k \ \matGamma_k e^{l_k\matLambda_k}\  \vecu_k & 
    &= \matbUm_k \ ( \veck_k + \vecpi_k ) \\
 \matU_k \ \matGamma_k e^{l_k\matLambda_k}\  \vecu_k &
    - \matD_{k+1} \ \matGamma_{k+1} e^{l_k\matLambda_{k+1}} \ \vecu_{k+1}
    &=  \matU_k \ \veck_k - \matD_{k+1} \ \veck_{k+1} \\
 \mattU_k \ \matGamma_k \matLambda_k e^{l_k\matLambda_k}\  \vecu_k &
    -\mattD_{k+1}\ \matGamma_{k+1} \matLambda_{k+1} e^{l_k\matLambda_{k+1}}\ \vecu_{k+1}
    &=  \mattU_k \ \veck_k - \mattD_{k+1}\ \veck_{k+1}
\end{array}$$
for any $k=1,\ldots,K$, with $\vecpi_k=\col[\pi_i, i\in\Jspace_k]$, that defining
\begin{eqnarray*}
\matA^D_k &=& \col[
 \matbDp_k \ \matGamma_k e^{l_{k-1}\matLambda_k},
 \matbUm_k \ \matGamma_k e^{l_k\matLambda_k},
 \matU_k \ \matGamma_k e^{l_k\matLambda_k},
 \mattU_k \ \matGamma_k \matLambda_k e^{l_k\matLambda_k}
], \\
\matA^U_k &=& \col[
 \matO_{k+1},
 \matO_{k+1},
 \matD_{k+1} \ \matGamma_{k+1} e^{l_k\matLambda_{k+1}},
 \mattD_{k+1}\ \matGamma_{k+1} \matLambda_{k+1} e^{l_k\matLambda_{k+1}}
], \\
\vecb_k &=& \col[
    \matbDp_k \veck_k,\
    \matbUm_k (\veck_k+\vecpi_k),\
    \matU_k \veck_k - \matD_{k+1} \veck_{k+1},\
    \mattU_k \vec0_k
],
\end{eqnarray*}
can be rewritten as
$$\matA^D_k \, \vecu_k - \matA^U_k \, \vecu_{k+1} = \vecb_k \ .$$

Defining the block matrix $\matA$ whose block diagonal and upper block diagonal are made respectively 
of $\matA^D_k$ and $\matA^U_k$ with $k=1,\ldots,K$,
all the unknowns may be solved by resolving the linear system 
$\matA \vecu = \vecb$, with $\vecu = \col[\vecu_k,k=1,\ldots,K]$, and 
$\vecb = \col[\vecb_k,k=1,\ldots,K]$.

The assumption that $\Jspacep\cup\Jspacem$ is not empty and that we can choose
$i\in\Jspacep\cup\Jspacem$ such that $a(i)<b(i)$ implies that 
the matrix $\matA$ has dimension at least one. 
In addition, considering the interval $I_k$ such that 
$a(i)=l_{k-1}$ if $i\in\Jspacep$ (corr. $b(i)=l_k$ if $i\in\Jspacem$),
the corresponding $k$-block in the matrix $\matA$  contains a strictly positive submatrix
$\matbDp_k \matGamma_k e^{l_{k-1} \matLambda_k}$ 
(resp. $\matbUm_k \matGamma_k e^{l_k \matLambda_k}$ ).
Having that all other submatrices in the same rows are all zeros, we deduce that the matrix $\matA$ is invertible if the rank of the $k$-block column is maximum. It is then easy to realize by induction that $\matA$ is invertible if for all $k=1,\ldots,K$ the corresponding
$k$-block column has maximum rank.

Consider the $k$ block, in order to prove that its rank is maximum we can look at a square submatrix
with its dimension equals to the size of the block. Noticing that the matrices 
$\matDp_k$ (rep. $\matUm_k$) is a submatrix of $\matD_k$ (resp $\matU_k$) that
corresponds to the complementary states in $\Jspacep \cap \Jspace_k$ (resp. $\Jspacem \cap
\Jspace_k$) not selected by $\matbDp$ (resp. $\matbUm$), after reordering the rows we get that the
sought square submatrix is given by
$$\matA_k = 
\left(\begin{array}{ll}
  e^{l_{k-1} \matLambda_k^+} & \matE_k^+ \matP_k \matGamma^-_k  e^{-l_{k-1} \matLambda_k^-}\\
  \matE_k^-  \matP_k  \matGamma^+_k e^{l_k \matLambda_k^+} & e^{-l_k \matLambda_k^-}
\end{array}\right)\ .$$

The matrix $\matA_k$ is invertible if its Schur complement is.
It is given by     
$$e^{l_{k-1} \matLambda_k^+} \Big( \matI_k
- e^{-l_{k-1} \matLambda_k^+} \matE_k^+ \matP_k \matGamma^-_k  e^{-l_{k-1} \matLambda_k^-}
e^{-l_k \matLambda_k^+}
 \matE_k^- \matP_k \matGamma^+_k  e^{-l_k \matLambda_k^-}
\Big)
$$
and it is invertible by the Levy-Desplanques theorem (see Lemma B.1 in \cite{seneta:1981})
after noticing that all the matrices in the
last term of the equation above are substochastic with at least one strictly substochastic.
%%%%%%%%%%%%%%%%%%%%%%%%%%%%%%%%%%%%%%%%%%%
%%%%%%%%%%%%%%%%%%%%%%%%%%%%%%%%%%%%%%%%%%%
    %************************************************************************
%	Fluid queues with modulated buffer. Special system |E|=2
%************************************************************************
\section{Fluid queues with modulated buffer}\label{sec:fluid-queues}
%************************************************************************
One direct application of the model presented in the previous sections is the case of the
fluid queue with Markov modulated buffer.
Indeed if we assume that for all states the lower barrier is zero, 
i.e. $a(i)=0$ for any $i\in\Jspace$,
we can look at the process $Z$ as the buffer content of a fluid queue
whose net-input flows is the process $X(t)$ defined in (\ref{eq:X.proc}) 
and whose buffer is equal to $b(i)$ when the environment $J$ is in state $i\in\Jspace$.

In this simplified setting, it is easier to solve the system (\ref{eq:sys.diff.eq}),
as it is possible without loss of generality to order the buffer levels
in increasing order, i.e. $b(1)\leq b(2) \ldots \leq b(N)$.

While the results of previous section are completely general and allow in any case to compute
at least numerically the stationary distribution for any configuration of the barriers and the 
parameters of the Markov modulated Brownian motion, we limit the present discussion only to some
few cases where it is possible to write the general solution in an easy form.

It is worth mentioning that a closed form solution can be obtained for the general case  when $N=2$, but the expression is cumbersome and we prefer not to include it.
It requires solutions of forth-order algebraic equations, see Section~\ref{div.payout.ex} for an example of this type but applied to the dividend-payout problem.
For this reason we further limit our focus to more specific cases and provide the
solution for the case where the net-input flow is not modulated, i.e. the modulation 
applies only to the buffer levels, and for the case where only in one of the two states the diffusion coefficient is positive.
%%%%%%%%%%%%%%%%%%%%%%%%%%%%%%%%%%%%%%%%%%%
%%%%%%%%%%%%%%%%%%%%%%%%%%%%%%%%%%%%%%%%%%%
    \subsection{Two sided Markov modulated reflection of a $(\mu,\sigma)$-Brownian motion}
\label{BM-Queue}
In this section we analyze the case when $\Jspace=2$ and the drift and diffusion components, 
$\mu$ and $\sigma$, do not depend on the environment. The reflecting levels are given by
$a(1)=a(2)=0$ and $0<b(1)<b(2)$, and the system can be looked at as a fluid queue that for
exponential periods of times, i.e. when $J=1$, its buffer is reduced to a smaller value. 
A typical application for this could be a service station that for specified period of times
receives help from another service station with which it shares the buffer size. 
When the second system turns on, the buffer of the first station is reduced and the overflow fluid
becomes the starting content for the second station.
In the following we derive the stationary distribution of the system, 
and in the next subsection we compute the discontinuity rate of the content process
together with the size distribution of the discontinuities. 

The system (\ref{eq:sys.diff.eq}) can be solved by considering the two
intervals $I_1=[0,b(1)]$ and $I_2=[b(1),b(2)]$. For the interval $I_1$ we have that
\begin{equation}
\left\{\begin{array}{ll}
 \frac{1}{2} \sigma^2 \, \Pi_1''(z) - \mu \, \Pi'_1(z) 
    - q_{12} \, \Pi_1(z) + q_{21} \, \Pi_2(z) =& 0 \\ 
    \\		
 \frac{1}{2} \sigma^2 \, \Pi_2''(z) - \mu \, \Pi'_2(z) 
    + q_{12} \, \Pi_1(z) - q_{21} \, \Pi_2(z) =& 0 
\end{array} \right. 
\end{equation}
and for the interval $I_2$ we have 
\begin{equation}\label{eq:E=2.above.b1}
\hspace{1.15cm} \frac{1}{2} \sigma^2 \, \Pi''_2(z) - \mu \, \Pi'_2(z) 
    - q_{21} \, \Pi_2(z) = - q_{12} \pi_1
\end{equation}
In addition we have the boundary conditions $\Pi_i(0)=0$ and
$\Pi_i(b(i))=\pi_i=q_{3-i,i}/(q_{12}+q_{21})$, $i=1,2$, the continuity condition
$\Pi_2(b(1)-)=\Pi_2(b(1)+)$ and the differentiability condition
$\Pi'_2(b(1)-)=\Pi'_2(b(1)+)$.

Let $\matLambda_1^\pm = \matTheta_1 \pm \matDelta_1$ be matrix solutions of the equation 
$ \frac{1}{2} \sigma^2 \, \matL^2 \mp \mu \, \matL + \matQ = \matO$. In this case
since $\Jspace^+_1=\Jspace^-_1=\Jspace$ we have that $\matGamma^+_1=\matGamma^-_1=\matI$,
in addition $\matP_1=\diag[\pi_1,\pi_2]$. %
When $\mu < 0$ the matrices $\matLambda_1^\pm$
below are substochastic and are the same ones appearing in (\ref{eq:interval.sol}).
When $\mu > 0$ they are not substochastic but the result still holds. 
When $\mu=0$ it follows that the asymptotic drift $\kappa=0$
and the root $0$ has multiplicity $2$. As already mentioned before, for this case
the solution looks slightly different from the one given below, and it 
can be computed using similar arguments. We have decided to omit its expression here.

Let $\lambda_1^\pm$ be the negative solutions of the equations 
$ \frac{1}{2} \sigma^2 \, z^2 \mp \mu \, z - (q_{12}+q_{21}) = 0$.
Define 
$\Theta_1=(\lambda_1^+ + \lambda_1^-)/2=-(\sqrt{\mu^2+2 (q_{12}+q_{21}) \sigma^2})/\sigma^2$ 
and
$\Delta_1=(\lambda_1^+ - \lambda_1^-)/2 = \Delta = \mu/\sigma^2$,
then we have that
$$
\matTheta_1 = 
\Theta_1 \left(
\begin{array}{cc}
   \pi_2 &  - \pi_2 \\
 - \pi_1 &    \pi_1 
\end{array}
\right)
+
\Delta
\left(
\begin{array}{cc}
 \pi_1 & \pi_2 \\
 \pi_1 & \pi_2
\end{array}
\right) \ ,
$$
and $\matDelta_1 = \Delta \matI$.
Considering the continuity conditions at $z=0$ and $z=b(1)$ we can write the solution 
for $z \in I_1$ as
% $$\left(\begin{array}{l}\Pi_1(z) \\ \Pi_2(z)\end{array}\right) = 
%  e^{-\matLambda_1^- \, z} \, \vecu_1^{\, -} + e^{\matLambda_1^+ \, z} \, \vecu_1^{\, +} \ .$$
\begin{equation}
\left(\begin{array}{l}\Pi_1(z) \\ \Pi_2(z)\end{array}\right) = 
  e^{\Delta \, (z-b(1))} \, \matP_1 \, \sinh(\matTheta_1 \, z) \, 
  \sinh^{-1}(\matTheta_1 \, b(1)) \, \matP_1^{-1}
	\left(\begin{array}{c}\pi_1 \\ \Pi_2(b(1))\end{array}\right) % bound. conds. at b(1)
\ .
\end{equation}
\begin{myc3}
 $$
\left(\begin{array}{cc} \pi_2 & - \pi_2 \\ - \pi_1 & \pi_1 \end{array}\right)
    =
    \left(\begin{array}{cc} 1 & - q_{12}/q_{21} \\ 1 & 1 \end{array}\right)
    \left(\begin{array}{cc} 0 & 0 \\ 0 & 1 \end{array}\right)
    \left(\begin{array}{cc} \pi_1 & - \pi_2 \\ - \pi_1 & \pi_1 \end{array}\right)
$$
\end{myc3}

Let $\lambda_2^\pm$ be the negative solutions of the equations 
$ \frac{1}{2} \sigma^2 \, z^2 \mp \mu \, z - q_{21} = 0$.
Define 
$\Theta_2=(\lambda_2^+ + \lambda_2^-)/2=-(\sqrt{\mu^2+2 q_{21} \sigma^2})/\sigma^2$ 
and
$\Delta_2=(\lambda_2^+ - \lambda_2^-)/2 = \Delta = \mu/\sigma^2$,
we can write the solution for $z \in I_2$ with boundary condition $\Pi_2(b(2))=\pi_2$ as
%$$\Pi_2(z) =  e^{-\lambda_2^- \, z} \, u_2^- + e^{\lambda_2^+ \, z} \,  u_2^+ + \pi_2\ .$$
\begin{equation}\label{eq:E2.I1}
\Pi_2(z) = \pi_2 - e^{\Delta(z-b(1))} \frac{\sinh(\Theta_2(b(2)-z))}{\sinh(\Theta_2(b(2)-b(1)))}
	(\pi_2-\Pi_2(b(1))) \ . 
\end{equation}
In the equation above we already used the continuity condition for $z=b(1)$ by imposing the
boundary value at $b(1)$ to be equal to $\Pi_2(b(1))$.

The solution is completely identified by solving for the value of $\Pi_2(b(1))$ that assures 
the differentiability of $\Pi_2(z)$ at $z=b(1)$.

We have that, for $z \in I_1$
$$
\vec \Pi'(z) = \phantom{-} \Delta \, \vec \Pi(z) +
  e^{\Delta \, (z-b(1))} \, \matP_1 \, \matTheta_1 \, \cosh(\matTheta_1 \, z) \, 
  \sinh^{-1}(\matTheta_1 \, b(1)) \, \matP_1^{-1} \, \vec \Pi(b(1)) \ , % bound. conds. at b(1)
$$
 and for $ z \in I_2$
$$
\Pi'_2(z) = -\Delta \, (\pi_2 - \Pi_2(z))%
    + e^{\Delta(z-b(1))} \Theta_2 %
	\frac{\cosh(\Theta_2(b(2)-z))}{\sinh(\Theta_2(b(2)-b(1)))}%
	(\pi_2-\Pi_2(b(1))) \ .
$$
To simplify the resulting expression define 
$k_i = \vec e^{\, \t}_2 \,\matP_1 \matTheta_1 \coth(\matTheta_1 \, b(1)) \matP_1^{-1}
\, \vec e_i$, $i=1,2$, where $\vec e_i$ is the $i$-th column vector of the canonical base
in $\R^2$, and $k_3= \Theta_2 \coth(\Theta_2(b(2)-b(1)))$,
then the value of $\Pi_2(z)$ at the boundary $b(1)$ is given by
\begin{equation}
 \Pi_2(b(1)) = \frac{\pi_2(k_3-\Delta)-\pi_1 k_1}{k_2 + k_3}
\end{equation}
%%%%%%%%%%%%%%%%%%%%%%%%%%%%%%%%%%%%%%%%%%%%%%%%%%%%%%%%%%%%%%%%
	%************************************************************************
%		The time up to the first jump			
%************************************************************************
\subsubsection{Mean inter-regeneration time and distribution of the content jump}
In this section we continue the analysis of the previous section and we are going to study the
process only at the moment of the discontinuities, that is when the environment jumps
from state $2$ to state $1$ and the buffer content just before the jump epoch is greater than
the buffer level $b(1)$. 

These special epochs are regeneration points for the process $Z(t)$, and we denote by $\tau$ the general interval of time that lags betweens two such discontinuities of $Z(t)$. We are going to exploit the regenerative structure of $Z(t)$ at its discontinuity points to determine 
the expectation of $\tau$ and the distribution of $Z(\tau-)$. 

This quantities may have relevance when studying an optimization problem. For example, if we consider the example mentioned before, the discontinuity may model the start-up content of a second server to which there can be associated start-up costs taken into account by a given cost function. The value of $1/\E[\tau]$ gives the average rate at which such costs incur.

Assume that $Z(0-)>b(1)$ and $Z(0)=b(1)$ and define $\tau=\inf_{t>0} \{ \Delta Z(t) < 0 \}$,
the picture below helps in understanding the given definition.
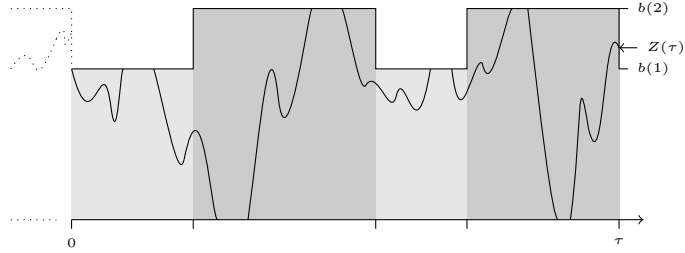
\begin{figure}[ht]%
\begin{center}%
%\documentclass{amsart}
%\usepackage{tikz}
%
%\begin{document}
%
\begin{tikzpicture}[x=4mm,y=4mm]
		%% background
    \fill [fill=black!10!white] (0,0) rectangle (4,5);
    \fill [fill=black!10!white] (10,0) rectangle (13,5);
    \fill [fill=black!20!white] (4,0) rectangle (10,7);
    \fill [fill=black!20!white] (13,0) rectangle (18,7);
		%% b=5, B=7, T1b=4, T1B=6, T2b=3, T2B=5
	\draw [dotted] (-2,7) -- (0,7) -- (0,5);
	\draw [dotted] (-2,0) -- (-0.5,0);
    \draw (0,5) -- (4,5) -- (4,7) -- (10,7) -- (10,5) -- (13,5) -- (13,7) -- (18,7) -- (18,5);
	\draw [->] (0,0) -- (18.8,0);
		%% before time 0
	\draw [dotted] (-2,5) .. controls (-1,6.5) and (-1.5,3.5) .. (-0.5,6);
	\draw [dotted] (-0.5,6) .. controls (0,7) and (-0.45,5) .. (0,6.2);
	%% after time 0
    \draw (0,5) .. controls (0.7,2.2) and (1,6) .. (1.3,3.7);
	\draw (1.3,3.7) .. controls (1.5,2.2) and (1.6,5) .. (1.7,5);
		%% hit boundary b
	\draw [very thick](1.7,5)(2.7,5);
	\draw (2.7,5) .. controls (3,4) and (3.5,1.1) .. (3.7,2);
	\draw (3.7,2) .. controls (4.3,5) and (4.6,0) .. (4.8,0);
		%% hit boundary 0
	\draw [very thick](4.8,0)(5.8,0);
	\draw (5.8,0) .. controls (6.2,3) and (6.5,6.8) .. (6.8,4);
	\draw (6.8,4) .. controls (7.1,1.6) and (7.8,7) .. (7.9,7);
		%% hit boundary B
	\draw [very thick](7.9,7)(8.9,7);
	\draw (8.9,7) .. controls (9.2,6) and (9.5,4) .. (9.7,4.5);
	\draw (9.7,4.5) .. controls (10.1,5) and (10.5,3.2) .. (10.7,4.2);
	\draw (10.7,4.2) .. controls (11,5) and (11.3,2) .. (11.8,5);
		%% hit boundary b
	\draw [very thick](11.8,5)(12.5,5);
	\draw (12.5,5) .. controls (12.8,2.5) and (13.2,5.4) .. (13.5,5.2);
	\draw (13.5,5.2) .. controls (13.8,3.8) and (14.3,7) .. (14.5,7);
		%% hit boundary B
	\draw [very thick](14.5,7)(15,7);
	\draw (15,7) .. controls (15.4,3.5) and (15.8,0) .. (16,0);
		%% hit boundary 0
	\draw [very thick](16,0)(16.4,0);
	\draw (16.4,0) .. controls (16.8,3) and (16.6,5.5) .. (17,3);
	\draw (17,3) .. controls (17.5,1) and (17.5,7) .. (18,5.7); %% hit - \tau
		%% x-ticks
		  \draw (0,0) -- (0,-0.3) node [below] {\tiny $0$};
		  \draw (4,0) -- (4,-0.3);
		  \draw (10,0) -- (10,-0.3);
		  \draw (13,0) -- (13,-0.3);
		  \draw (18,0) -- (18,-0.3) node [below] {\tiny $\tau$};
		%% dotted lines
		  %\draw [dotted] (0,0) -- (0,5);
		  %\draw [dotted] (18,0) -- (18,5);
		%% y-ticks
		  \draw (18,7) -- (18.3,7) node [right] {\tiny $b(2)$};
		  \draw (18,5) -- (18.3,5) node [right] {\tiny $b(1)$};
		  \draw [<-] (18,5.7) -- (18.6,5.7) node [right] {\tiny $Z(\tau)$};
    \end{tikzpicture}
%
%\end{document} 
%
\end{center}%
\end{figure}

Let $f(z)$ be any twice differentiable function, by \Ito's Lemma % \cite{mckean:1969}
we have that
\begin{eqnarray}
    f(Z(t)) = f(Z(0)) &+& \sigma \int_0^t f'(Z(s)) dW(s) 
		+ \int_0^t \frac{1}{2}\sigma^2 f''(Z(s)) + \mu f'(Z(s)) \, ds \nonumber\\
		&+& f'(0) \, L(t) - f'(b(1)) \, U_1(t) - f'(b(2)) \, U_2(t)
\label{eq:Ito-formula.2states}
\end{eqnarray}
where $U_i(t)$, $i\in\Jspace$ is the amount of content lost at the barrier $b(i)$ during the
interval $[0,t)$. %
Evaluating (\ref{eq:Ito-formula.2states}) at $t=\tau$ under the condition that $Z(0)=b(1)$
and taking expectation we get that
\begin{eqnarray}
    \E_{b(1)}[f(Z(\tau))-f(b(1))] 
	&=& \E_{b(1)}\left[\int_0^\tau \frac{1}{2}\sigma^2 f''(Z(s)) 
	    + \mu f'(Z(s)) \, ds \right] \label{eq:average.at.tau}  \\
	&\phantom{=}&   + f'(0) \, \E_{b(1)}[L(\tau)] - f'(b(1)) \, \E_{b(1)}[U_1(\tau)] \nonumber
\\
	&\phantom{=}&   - f'(b(2)) \, \E_{b(1)}[U_2(\tau)] \phantom{\Big[}\nonumber 
\end{eqnarray}

Having that $\tau$ is a regeneration point for $Z$ we have by renewal theory that
$$ 
\Pi(z) 
= \Pr\{Z^* \leq z\} 
= \eta \, \E_{b(1)} \left[ \int_0^\tau 1_{\{Z(s)\leq z\}} \, ds\right],
$$ 
where we defined $\eta=1/E_{b(1)}[\tau]$, the rate of discontinuities of $Z$.

In general we have that 
\begin{equation}\label{eq:renewal}
\E[f(Z^*)] = \int_0^{b(2)} f(z) \, \Pi(dz) = 
\eta \, \E_{b(1)} \left[ \int_0^\tau f(Z(s)) \, ds\right] \ ,
\end{equation}
so that multiplying equation (\ref{eq:average.at.tau}) by $\eta$, applying (\ref{eq:renewal})
and using repeatedly integration by parts similarly to the proof of Theorem \ref{thm:stat.dist},
we get the following system of differential equations 
by collecting all the integrand terms with factor $f(z)$,
\begin{equation}\label{eq:diff.eq}
    \frac{1}{2} \sigma^2 \, \pi''(z) - \mu \, \pi'(z)	= \eta \, H'(z)  \,
     1_{\{b(1) < z \leq b(2)\}}
\end{equation}
where $H(z) = \Pr\{b(1) \leq Z(\tau-) \leq z\}$, with $b(1) < z \leq b(2)$, and
$\pi(z)=\Pi'(z)$.

Having that $\Pi(z)=\Pi_2(z) + \Pi_1(z \wedge b(1))$, by comparing the second equation 
in (\ref{eq:diff.eq}) with the derivative of equation (\ref{eq:E=2.above.b1}) we get that
\begin{equation}\label{eq:eta.H}  \eta \, H'(z) = q_{21} \, \Pi'_2(z)  \ .\end{equation}
Integrating last equation over the interval $[b(1),b(2)]$ with boundary conditions $H(b(1))=0$ and
$H(b(2))=1$ 
it follows that $  \eta  = q_{21} \Big(\pi_2 - \Pi_2(b(1)) \Big)$. 
Substituting its value in (\ref{eq:eta.H}) and integrating we get that, for 
$z \in I_2$,
$$  H(z) 
    =  \frac{\Pi_2(z)-\Pi_2(b(1))}{\pi_2 - \Pi_2(b(1))} 
    = 1- e^{\Delta(z-b(1))} \frac{\sinh(\Theta_2(b(2)-z))}{\sinh(\Theta_2(b(2)-b(1)))} 
 \ ,$$
where we used the expression of $\Pi_2(z)$ given in (\ref{eq:E2.I1}).

%%%%%%%%%%%%%%%%%%%%%%%%%%%%%%%%%%%%%%%%%%%
%%%%%%%%%%%%%%%%%%%%%%%%%%%%%%%%%%%%%%%%%%%
%%   \input{General_Case_E_equal_2}
%%%%%%%%%%%%%%%%%%%%%%%%%%%%%%%%%%%%%%%%%%%
%%%%%%%%%%%%%%%%%%%%%%%%%%%%%%%%%%%%%%%%%%%
    \subsection{Two-state modulation with at least one with no diffusion component}
In this section we consider the case when $N=2$ and only in one of the two states
the diffusion coefficient is positive. 

The way to proceed to compute the stationary distribution is similar to the one used in the
previous section, but we decided to include these examples because for these cases the matrices 
$\Gamma^\pm$ are rectangular and do not reduce to the identity matrix.
In \cite{dauria:ivanovs:kella:mandjes:2010} and \cite{dauria:ivanovs:kella:mandjes:2010b}
the authors looked at how compute the stationary distribution for the case of two-side reflection
with two non-modulated barriers, but no explicit examples where given and as far as we know
they are not treated elsewhere.

The system (\ref{eq:sys.diff.eq}) can be solved by considering the two
intervals $I_1=[0,b(1)]$ and $I_2=[b(1),b(2)]$. We consider the two cases when
the state with no diffusion component is the first one, i.e. $\sigma_1=0$ and
then when the state with no diffusion is the second one, i.e. $\sigma_2=0$.

In the second case in order to have positive probability for the process to enter the interval
$(b(0),b(1)]$ we assume that $\mu_2>0$. To simplify the analysis and reduce the cases 
we assume that the asymptotic drift $\kappa<0$ and for the first case that $\mu_1<0$.

\begin{description}
%%%%%%%%%%%%%%%%%%%%%%%%%%%%%%%%%%%%%%%%%%%%%%%%%%%%%
%%%%%%%%%%%%%%%%%%%%%%%%%%%%%%%%%%%%%%%%%%%%%%%%%%%%%
%%%%%%%%%%%%%%%%%%%%%%%%%%%%%%%%%%%%%%%%%%%%%%%%%%%%%
\item[Case $\mu_1<0, \, \sigma_1=0;$ $\sigma_2>0;$ $\kappa<0$:]
For $z \in I_1$ we have that
$$
\left\{\begin{array}{ll}
 \phantom{\frac{1}{2} \sigma_1^2 \, \Pi_1''(z)} - \mu_1 \, \Pi'_1(z) 
    - q_{12} \, \Pi_1(z) + q_{21} \, \Pi_2(z) =& 0 \\ 
    \\		
 \frac{1}{2} \sigma_2^2 \, \Pi_2''(z) - \mu_2 \, \Pi'_2(z) 
    + q_{12} \, \Pi_1(z) - q_{21} \, \Pi_2(z) =& 0 
\end{array} \right. 
$$
and for $ \in I_2$ we have 
$$
\hspace{1.15cm} \frac{1}{2} \sigma_2^2 \, \Pi''_2(z) - \mu_2 \, \Pi'_2(z) 
    - q_{21} \, \Pi_2(z) = - q_{12} \pi_1 \ .
$$
In addition we have the boundary conditions $\Pi_i(0)=0$ and
$\Pi_i(b(i))=\pi_i=q_{3-i,i}/(q_{12}+q_{21})$, $i=1,2$, the continuity condition
$\Pi_2(b(1)-)=\Pi_2(b(1)+)$ and the differentiability condition
$\Pi'_2(b(1)-)=\Pi'_2(b(1)+)$.

Having $\mu_1<0$ it follows that $E^+=\{2\}$ and $E^-=\{1,2\}$. Hence defining 
$\gamma_1^+ = \Pr_1\{\sup_{t>0}\{X(t)\}=0\}$ we have that
$\matGamma_1^+ = (\gamma_1^+,1)^\t$ and $\matGamma_1^- = \matI^-$.

Solving the system (\ref{eq:vec.sys.eq.Ik}) that is explicitly rewritten as
$$
\left\{
 \begin{array}{ll}
\phantom{-}  q_{12} (1 - \gamma_1^+) - \mu_1 \,\, \lambda_1^+ \gamma_1^+ &= 0 \\ \\
- q_{21} (1 - \gamma_1^+) - \mu_2 \, \lambda_1^+ + \frac{1}{2}\sigma_2^2 \, (\lambda_1^+)^2 &=0
 \end{array}\right.
$$
and selecting the solution with $\lambda_1^+<0$ we have that 
$$
\lambda_1^+	= \frac{\mu_2}{\sigma_2^2}-\frac{q_{12}}{\mu_1}
    -\frac{1}{\mu_1\,\sigma_2^2} \, \delta_1 ; \quad
\gamma_1^+	= -\frac{1}{2}\frac{\mu_2}{\mu_1}\frac{q_{12}}{q_{21}}
    -\frac{q_{12}^2 \sigma_2^2}{4\,q_{21}\mu_1^2}
    -\frac{1}{2 \mu_1^2}\frac{q_{12}}{q_{21}} \, \delta_1 \nonumber
$$
where $\delta_1 = \sqrt{\mu_1^2 \mu_2^2+2 \mu_1 \mu_2 q_{12} \frac{\sigma_2^2}{2}
		+4 \mu_1^2 q_{21}\frac{\sigma_2^2}{2} + q_{12}^2 \frac{\sigma_2^4}{4}}$.
Solving the system (\ref{eq:vec.sys.eq.Ik}) that is explicitly rewritten as
$$
\left\{
 \begin{array}{ll}
  q_{12} + \mu_1 \lambda_{12}^- &= 0 \\ \\
  q_{21} + \mu_2 \lambda_{21}^- - \frac{1}{2} \sigma_2^2 \, \lambda_{12}^- \, \lambda_{21}^-
    -\frac{1}{2} \sigma_2^2 (\lambda_{21}^-)^2 &=0
 \end{array}\right.
$$
we get the following solutions 
$$
\lambda_{12}^- = -\frac{q_{12}}{\mu_1}; \quad \quad
\lambda_{21}^- = \frac{\mu_2}{\sigma_2^2} + \frac{q_{12}}{2\mu_1}
	+ \frac{1}{\mu_1\,\sigma_2^2} \, \delta_2
$$
where $\delta_2 = \sqrt{2 q_{21} \, \mu_1^2 \, \sigma_2^2 
		    + \left[ \mu_1 \, \mu_2 + q_{12} \frac{1}{2} \sigma_2^2 \right]^2}$.

Over the interval $I_1$, the stationary distribution is given by
$$\left(\begin{array}{l}\Pi_1(z) \\ \Pi_2(z)\end{array}\right) = 
  \matP_1 \, \matGamma_1^+ e^{ z \, \lambda_1^+} \, c_1^+ 
+ \matP_1 \, \matGamma_1^- e^{-z \, \matLambda_1^-}\, \vecc_1^{\, -} \ .$$
Having $\Pi_2(0)=0$ we can solve for $c_1^+ = -c_{12}^-$ 
and with $\Pi_1(b(1))=\pi_1$ we have
\begin{equation}\label{eq:sol.I0.case1}
\left(\begin{array}{l}\Pi_1(z) \\ \Pi_2(z)\end{array}\right) = 
 \matP_1 
    \, \left[ \hat{\matGamma}_1^+ e^{ z \, \matLambda_1^+}
    - \matGamma_1^- e^{-z \, \matLambda_1^-} \right] 
    \, \matC ^{-1}
 \matP_1^{-1}
    \left(\begin{array}{c} \pi_1 \\ \Pi_2(b(1))\end{array}\right) \ ,
\end{equation}
where $\hat{\matGamma}_1^+ = \left(\begin{array}{ll} 0 & \gamma_1^+ \\ 0 & 1 \end{array}\right)$,
$\matLambda_1^+=\lambda_1^+ \, \matI$ and 
$\matC =  \left[ \hat{\matGamma}_1^+ e^{ b(1) \, \matLambda_1^+}
    - \matGamma_1^- e^{-b(1) \, \matLambda_1^-} \right] $.

Over the interval $I_2$ let $\lambda_2^\pm$ be the negative solutions of the equations 
$ \frac{1}{2} \sigma_2^2 \, z^2 \mp \mu_2 \, z - q_{21} = 0$.
Define 
$\Theta_2=(\lambda_2^+ + \lambda_2^-)/2=-(\sqrt{\mu_2^2+2 q_{21} \sigma_2^2})/\sigma_2^2$ 
and
$\Delta_2=(\lambda_2^+ - \lambda_2^-)/2 = \mu_2/\sigma_2^2$,
we can write the solution for $z \in I_2$ with boundary condition $\Pi_2(b(2))=\pi_2$ as
\begin{equation}\label{eq:sol.I1.case1}
\Pi_2(z) = 
\pi_2 - e^{\Delta_2(z-b(1))} 
    \frac{\sinh(\Theta_2(b(2)-z))}{\sinh(\Theta_2(b(2)-b(1)))} (\pi_2-\Pi_2(b(1))) \ .
\end{equation}
In the equation above we already used the continuity condition for $z=b(1)$ by imposing the
boundary value at $b(1)$ to be equal to $\Pi_2(b(1))$.

The solution is completely identified by solving for the value of $\Pi_2(b(1))$ that assures 
the differentiability of $\Pi_2(z)$ at $z=b(1)$.

We have that for $z \in I_1$
$$
\vec \Pi'(z) = \matP_1 
    \, \left[ \hat{\matGamma}_1^+ \, \matLambda_1^+ \, e^{ z \, \matLambda_1^+}
    + \matGamma_1^- \, \matLambda_1^- \, e^{-z \, \matLambda_1^-} \right] 
    \, \matC ^{-1}  \matP_1^{-1}
    \left(\begin{array}{c} \pi_1 \\ \Pi_2(b(1))\end{array}\right) % bound. conds. at b(1)
$$
and for $z \in I_2$ 
$$
\Pi'_2(z) = -\Delta_2 \, (\pi_2 - \Pi_2(z))
    + e^{\Delta_2 (z-b(1))} \Theta_2 
	\frac{\cosh(\Theta_2(b(2)-z))}{\sinh(\Theta_2(b(2)-b(1)))}
	(\pi_2-\Pi_2(b(1))) \ .
$$
To simplify the resulting expression define 
$$k_i = \vec e^{\, \t}_2 \,
\matP_1 
    \, \left[ \hat{\matGamma}_1^+ \, \matLambda_1^+ \, e^{ b(1) \, \matLambda_1^+}
    + \matGamma_1^- \, \matLambda_1^- \, e^{-b(1) \, \matLambda_1^-} \right] 
    \, \matC^{-1}  \matP_1^{-1}
\, \vec e_i, \quad i=1,2 \ ,$$
where $\vec e_i$ is the $i$-th column vector of the canonical base
in $\R^2$, and 
$$k_3=  \Theta_2 \coth(\Theta_2(b(2)-b(1))) \ ,$$
then the value of $\Pi_2(z)$ at the boundary $b(1)$ is given by
\begin{equation}\label{eq:const.case1}
 \Pi_2(b(1)) = \frac{\pi_2(k_3-\Delta_2)-\pi_1 k_1}{k_2 + k_3} \ .
\end{equation}

Notice that in this case the function $\Pi_1(z)$ is not continuous at $z=0$ 
where it gives the probability to find the system empty when the environment 
is found in state $1$, i.e. 
$$\Pr(Z^*=0,J^*=1)=\Pi_1(0)=
\vec e^{\, \t}_1 \matP_1 
    \, \left[ \hat{\matGamma}_1^+ - \matGamma_1^- \right] 
    \, \matC^{-1} \matP_1^{-1}
    \left(\begin{array}{c} \pi_1 \\ \Pi_2(b(1))\end{array}\right) \ .$$
%%%%%%%%%%%%%%%%%%%%%%%%%%%%%%%%%%%%%%%%%%%%%%%%%%%%%
%%%%%%%%%%%%%%%%%%%%%%%%%%%%%%%%%%%%%%%%%%%%%%%%%%%%%
%%%%%%%%%%%%%%%%%%%%%%%%%%%%%%%%%%%%%%%%%%%%%%%%%%%%%
\item[Case $\sigma_1>0;$ $\mu_2>0, \, \sigma_2=0;$ $\kappa<0$:]
For the interval $I_1$ we have that
$$
\left\{\begin{array}{ll}
 \frac{1}{2} \sigma_1^2 \, \Pi_1''(z) - \mu_1 \, \Pi'_1(z) 
    - q_{12} \, \Pi_1(z) + q_{21} \, \Pi_2(z) =& 0 \\ 
    \\		
 \phantom{\frac{1}{2} \sigma_2^2 \, \Pi_2''(z)} - \mu_2 \, \Pi'_2(z) 
    + q_{12} \, \Pi_1(z) - q_{21} \, \Pi_2(z) =& 0
\end{array} \right. 
$$
and for the interval $I_2$ we have 
$$
 - \mu_2 \, \Pi'_2(z) 
    - q_{21} \, \Pi_2(z) = - q_{12} \pi_1 \ .
$$
In addition we have the boundary conditions $\Pi_i(0)=0$ and
$\Pi_i(b(i))=\pi_i=q_{3-i,i}/(q_{12}+q_{21})$, $i=1,2$ and the continuity condition
$\Pi_2(b(1)-)=\Pi_2(b(1)+)$.

Having $\mu_2>0$ it follows that $E^+=\{1,2\}$ and $E^-=\{1\}$.
In a similar way as in the previous case we have 
$\matGamma_1^+ = \matI^+$ and 
$\matGamma_1^- = (1,\gamma_2^-)^\t$ where
$\gamma_2^- = \Pr_2\{\inf_{t>0}\{X(t)\}=0\}$.
Solving the system (\ref{eq:vec.sys.eq.Ik}) that is explicitly rewritten as
$$
\left\{
 \begin{array}{ll}
  q_{12} - \mu_1 \lambda_{12}^+ - \frac{1}{2} \sigma_1^2 \, \lambda_{21}^+ \, \lambda_{12}^+
    -\frac{1}{2} \sigma_1^2 (\lambda_{12}^+)^2 &=0\\ \\
  q_{21} - \mu_2 \lambda_{21}^+ &= 0
 \end{array}\right.
$$
we get the following solutions 
$$
\lambda_{12}^+ = -\frac{\mu_1}{\sigma_1^2} - \frac{q_{21}}{2\mu_2}
	+ \frac{1}{\mu_2 \, \sigma_1^2} \delta_1 ; \quad \quad
\lambda_{21}^+ = \frac{q_{21}}{\mu_2}
$$
where $\delta_1 = \sqrt{2 q_{12} \, \mu_2^2 \,  \sigma_1^2 
		    + \left[ \mu_2 \, \mu_1 + q_{21} \frac{1}{2} \sigma_1^2 \right]^2}$.

Solving the system (\ref{eq:vec.sys.eq.Ik}) that is explicitly rewritten as
$$
\left\{
 \begin{array}{ll}
  - q_{21} (1 - \gamma_2^-) + \mu_1 \, \lambda_1^- 
	+ \frac{1}{2}\sigma_1^2 \, (\lambda_1^-)^2 &=0 \\ \\
  \phantom{-} q_{21} (1 - \gamma_2^-) + \mu_2 \, \lambda_1^- \, \gamma_2^- &= 0
 \end{array}\right.
$$
and selecting the solution with $\lambda_1^-<0$ we have that 
$$
\lambda_1^-	= \frac{-\mu_1}{\sigma_1^2}+\frac{q_{21}}{\mu_2}
    -\frac{1}{\mu_2\,\sigma_1^2} 
	\delta_2 ; \quad
\gamma_2^-	= -\frac{1}{2}\frac{\mu_1}{\mu_2}\frac{q_{21}}{q_{12}}
    -\frac{q_{21}^2 \sigma_1^2}{4\,q_{12}\mu_2^2}
    +\frac{1}{2 \mu_2^2}\frac{q_{21}}{q_{12}} \delta_2	
$$
where $\delta_2 = \sqrt{\mu_2^2 \mu_1^2 + 2 \mu_2 \mu_1 q_{21} \frac{\sigma_1^2}{2}
		+ 4 \mu_2^2 q_{12}\frac{\sigma_1^2}{2} + q_{21}^2 \frac{\sigma_1^4}{4}}$.

Over the interval $I_1$, the stationary distribution is given by
$$
\left(\begin{array}{l}\Pi_1(z) \\ \Pi_2(z)\end{array}\right) = 
   \matP_1 \, \matGamma_1^+ e^{ z \, \matLambda_1^+}\, \vecc_1^{\, +} 
+  \matP_1 \, \matGamma_1^- e^{-z \, \lambda_1^-} \, c_1^- \ .
$$
Having $\vec\Pi(0)=\vec0$ we can solve for $\vecc_1^{\, +}=-c_1^- \, \matGamma_1^-$
and knowing that $\Pi_1(b(1))=\pi_1$ we get
\begin{equation}\label{eq:sol.I0.case2}
 \left(\begin{array}{l}\Pi_1(z) \\ \Pi_2(z)\end{array}\right) = 
  \frac{\pi_1}{c_1} \matP_1 ( e^{-z \, \matLambda_1^-} - e^{ z \, \matLambda_1^+}) \matGamma_1^-
\end{equation}
where $\matLambda_1^-=\lambda_1^- \, \matI$ and  
$c_1 = \vec e^{\, \t}_1 \,  \matP_1 ( e^{-b(1) \, \matLambda_1^-} - e^{ b(1) \, \matLambda_1^+})
\matGamma_1^-$.

Over the interval $I_2$ the solution is simply
\begin{equation}\label{eq:sol.I1.case2}
\Pi_2(z) = \pi_2 - e^{-\frac{q_{21}}{\mu_2} (z-b(1))} 
    (\pi_2 - \Pi_2(b(1)))\ ,
\end{equation}
where $\Pi_2(b(1))$ is known by continuity from equation (\ref{eq:sol.I0.case2}) and we used
the fact that $\pi_2 = \pi_1 \, q_{12}/q_{21}$.

Notice that in this case the function $\Pi_2(z)$ is not continuous at $z=b(2)$ 
where it gives the probability to find the system saturated when the environment 
is found in state $2$, i.e.
$$\Pr(Z^*=b(2),J^*=2) = \pi_2-\Pi_2(b(2-)) 
    =e^{-\frac{q_{21}}{\mu_2} (z-b(1))} 
    (\pi_2 - \Pi_2(b(1))) \ .$$
\end{description}
%%%%%%%%%%%%%%%%%%%%%%%%%%%%%%%%%%%%%%%%%%%
%%%%%%%%%%%%%%%%%%%%%%%%%%%%%%%%%%%%%%%%%%%
%%%%%%%%%%%%%%%%%%%%%%%%%%%%%%%%%%%%%%%%%%%
%%%%%%%%%%%%%%%%%%%%%%%%%%%%%%%%%%%%%%%%%%%
    \section{Synchronized Barrier Strategies for Dividend Payout}\label{sec:dividend-payout}
In this section we look at an application of the model presented in Section \ref{sec.model}
to the problem of computing the expected dividend payouts of a company. 
We assume that the company profit fluctuates according to the Markov modulated Brownian motion
$(X(t),J(t))$ where $X(t)$ is as in (\ref{eq:X.proc}) and $J(t)$ is the environment process as
introduced in Section \ref{sec.model}.

The model is a direct extension of the one studied in \cite{gerber:shiu:2004} where for the profit
process it was chosen a free Brownian motion. the introduction of the external environment $J(t)$
from a modeling perspective helps in adapting the process to the case where the profit process
may depend on external environmental situation, such as for example seasonal-dependent activities.

We assume that the company select for each state of the environment $J(t)$ a barrier level $b(J(t))$
and decides to pay dividends as soon as its surplus process reaches that level. 
In this way the surplus process behaves exactly as the Markov Modulated two-sided reflected Brownian
motion $Z(t)$ up to the moment of the first hitting time of the lower barrier, i.e. $\tau = \inf\{t
\geq 0: Z(t) = 0\}$, that corresponds to the ruin time for the company.

The non-discounted total dividends paid up to time $\tau$ is directly given by $U(\tau)$,
that is the upper regulator process computed at the ruin epoch.
In general, given the discount rate $\delta>0$, this amount is given by
\begin{equation}
 U = \int_0^\tau e^{-\delta\,t} U(dt) \ .
\end{equation}

$U$ is a random quantity that depends on the path realization of the process $(X,J)$, as
well as the selected barriers, $\{b(j)\}_{j\in\Jspace}$, and the start-up condition of the company
$(z,j)=(Z(0),J(0))$. 
The aim of the company is to compute to then optimize the expected value of total 
discount dividend paid over its time-horizon. We denote this quantity by $V(z,j)$, when then
starts at time $0$ with initial capital $z$ and in an environment state $j$. The formal definition
is given by
\begin{equation}
 V(z,j) = \E[ U | Z(0) = z, J(0) = j] = \E_{(z,j)}[ U ] \ .
\end{equation}

In the following we show how to heuristically determine the system of differential equation that 
admits as solution the desired quantity $V(z,j)$, in the next sections we show how to get 
the same system of equations in a more rigorous way.

Assuming that the starting position $z$ is far from the barrier level $b(j)$ we can assume that
in a relatively short interval $\Delta t$ the surplus process does not reach the reflecting barrier
such that the following relation holds
\begin{eqnarray*}
\E_{(z,j)}\Big[ V\big(Z(\Delta t), J(\Delta t)\big) \Big] 
% &=& \E_{(z,j)}\Big[\E_{(Z(\Delta t),J(\Delta t))} \Big[ 
% 	\int_{\Delta t}^\tau e^{-\delta\,(y-\Delta t)} U(dy)
% 	    \Big] \Big] \\
% &=& e^{\delta\,\Delta t} \, \E_{(z,j)}\Big[\E_{(Z(\Delta t),J(\Delta t))} \Big[ 
% 	\int_0^\tau e^{-\delta\,y} U(dy) - \int_0^{\Delta t} e^{-\delta\,y} U(dy)
% 	    \Big] \Big] \\
&=& e^{\delta \, \Delta t} \, V(z,j) 
    - e^{\delta \, \Delta t} \, \E_{(z,j)}\Big[\int_0^{\Delta t} e^{-\delta \, t} U(dt)\Big] \ .
\end{eqnarray*}

In addition we have that
\begin{eqnarray}
\E_{(z,j)}\Big[\int_0^{\Delta t} e^{-\delta \, t} U(dt)\Big]
% &=& (1 + q_{jj} \, \Delta t) \, o(\Delta t) + \sum_{k\not=j} q_{jk} \, \Delta t 
%     [(z-b_k) \vee 0 + o(\Delta t)] + o(\Delta t)\\
&=& \sum_{k\in\Jspace} q_{jk} \, (z-b_k)^+ \, \Delta t + o(\Delta t) \nonumber
\end{eqnarray}
so that using the first order Taylor expansion of $e^{\delta \, \Delta t}$, we finally get 
$$
\E_{(z,j)}\Big[ V\big(Z(\Delta t), J(\Delta t)\big) \Big]
= V(z,j) + \delta \, V(z,j) \, \Delta t 
    - \sum_{k\in\Jspace} q_{jk} \, (z-b_j)^+ \, \Delta t + o(\Delta t) \ .
$$

On the other side using the It\^o formula it follows 
\begin{eqnarray}
V\big(Z(\Delta t), J(\Delta t)\big) 
&=&  V\big(Z(0), J(0)\big) 
    + \mu(J(0)) \, V'\big(Z(0), J(0)\big) \, \Delta t \nonumber \\
&\phantom{=}& + \frac{1}{2} \sigma^2(J(0)) \, V''\big(Z(0), J(0)\big) \, \Delta t \nonumber\\
&\phantom{=}& + \sum_{k\in\Jspace} q_{J(0)k} \, V(Z(0) \wedge b(k), k) \, 
		    \Delta t + o(\Delta t) \nonumber
\end{eqnarray}
and equating the two expressions above we get the following differential equations
\begin{equation}\label{eq:sys.v}
\frac{1}{2} \sigma^2(j) \, V''(z,j) 
    + \mu(j) \, V'(z,j) 
    - \delta \, V(z,j) 	+ \sum_k q_{jk} [V(z \wedge b(k),k)
			+ (z-b(k))^+ ]= 0 \ . 
\end{equation}

\subsection{Regularity of $V(z,j)$}
The derivation of equation (\ref{eq:sys.v}) has been done by implicitly assuming regularity
condition of the unknown function $V(z,j)$, i.e. that it has second derivative on the interval
$(0,b(j))$ with the exception of at most isolated points. In this section we prove that indeed
$V(z,j)$ does admits second derivative and again that it satisfies equation (\ref{eq:sys.v}).
In the following we assume that $\sigma(j) > 0$ as, according to (\ref{eq:sys.v}), in the case it
was zero we would need only the first derivative of $V(z,j)$. The treatment below can be easily
adapted to handle this case.

Define $T_b = \min_{j\in\Jspace} \{ \inf_{t\geq0} \{(Z(t),J(t)) = (b(j),j)\}\}$ and
$T_0 = \inf\{t \geq 0, X(t) = 0\}$
we have that
$$V(z,j)=\E_{(z,j)}[e^{-\delta \, T_b} \, (V(b(J(T_b)),J(T_b)) - \Delta X(T_b)) \, 1\{T_b < T_0\}]$$
with  $\Delta X(t)= X(t)- X(t-)$.

Define the stopping time $\tau(h) = T_{z-h} \wedge T_{z+h} \wedge T_J$ with
$T_J=\inf\{t>0,J(t-) \not= J(t)\}$ and $T_{z \pm h} = \inf\{t>0, X(t) = z \pm h\}$, we have that
\begin{eqnarray}\label{eq:V(z,j)}
V(z,j)
&=& \E_{(z,j)}[e^{-\delta \, \tau(h)} \, (V(X(\tau(h)),J(\tau(h)))+ \Delta X(\tau(h)))] \\
&=&  g_+(h) \, V(z+h,j) + g_-(h) \, V(z-h,j) 
    + \sum_{k\in\Jspace} \tilde V_J(h,k) \, \frac{q_{jk}}{q_j} \, g_J(h) \nonumber
\end{eqnarray}%
where $g_J(h) = \Pr_{(z,j)}\{\tau(h)=T_J \}$,  
$g_\pm(h) = \E_{(z,j)}[e^{-\delta \, \tau(h)}; \, \tau(h)=T_{z \pm h}]$,
 $q_j = \sum_{k\not=j} q_{jk}$, and  %
$$
\tilde V_J(h,k) 
= \E_{(z,j)}[e^{-\delta \,T_J} (V(X(T_J),k) + \Delta X(T_J)) |\tau(h)=T_J, J(T_J) = k] \ .
$$

To compute $g_J(h)$ we have that
$$
g_J(h)
= \Pr_{(z,j)}\{\sup_{0\leq t \leq T_J} |X(t)-z|<h \} 
= \Pr_0\{T_{|h|}>T_J \} 
= 1-\E_0[e^{-q_j \, T_{|h|}}] \ ,
$$
where $T_{|h|}=\inf_{t \geq 0}\{|Y(t)| = h\}$ is the hitting time of the
set $(-h,h)^c$ of the process $Y(t)$ that is a Brownian motion with drift $\mu(j)$ and
diffusion coefficient $\sigma(j)$. It is known, see \cite{chung:2002}, that
\begin{equation}\label{eq:Lapl.hit.time}
\E_0[e^{- \lambda \, T_{|h|}}] 
= 
{\cosh}\left(\frac{h \mu}{\sigma}\right)
{\rm sech}\left(\frac{h \mu \sqrt{2 \lambda + \mu^2}}{\sigma}\right)
= 
1 - \lambda \, h^2 + o(h^2) \ ,
\end{equation}
and from this it follows that 
$g_J(h) = 1-\E_0[e^{-q_j \, T_{|h|}}]  = \frac{q_j}{\sigma^2} h^2 + o(h^2)$.

To compute $g_\pm(h)$ we have that
\begin{eqnarray*}
g_\pm(h) 
&=& \E_{(z,j)}[\E[e^{-\delta \, \tau(h)}; \, \tau(h)=T_{z \pm h}|T_J]] \\
&=& \int_{t=0}^\infty 
	\E[e^{-\delta \, \tau(h)}; \, \tau(h)=T_{z \pm h}|T_J=t] \, q_j e^{-q_j \, t} 
	\, dt \\
&=& \int_{t=0}^\infty 
	\E[e^{-\delta \, \tau(h)}\E[1\{\tau(h)=T_{z \pm h}\}|\tau(h),\, T_J=t]|T_J=t] \, q_j e^{-q_j
\, t} 
	\, dt
\end{eqnarray*}
and since the exit location $X(T_{z+h} \wedge T_{z-h})$ is independent of the exit time
$T_{z+h} \wedge T_{z-h}$, see \cite{stern:1977}, assuming $\mu > 0$, we have that
$$ \E[1\{\tau(h)=T_{z \pm h}\}|\tau(h),\, T_J=t] =  c_\pm \, 1\{\tau(h) < t\} $$
where 
%$c_+ = \exp(\frac{2 \mu }{\sigma^2}\, h)/(1 + \exp(\frac{2 \mu }{\sigma^2}\,h)$ 
%and $c_- = 1/(1+\exp(\frac{2 \mu }{\sigma^2}\,h)$,
$c_\pm = \frac{1}{2} \pm \frac{1}{2}
 (\exp(\frac{2 \mu }{\sigma^2}\, h)-1)/(\exp(\frac{2 \mu }{\sigma^2}\, h)+1)$,
therefore 
\begin{eqnarray*}
g_\pm(h) 
&=& c_\pm \, \int_{t=0}^\infty 
	\E[e^{-\delta \, \tau(h)} \, 1\{\tau(h) < t\}|T_J=t] \, q_j e^{-q_j \, t} 
	\, dt \\
&=&  c_\pm \, \int_{t=0}^\infty  \int_{x=0}^t  
	e^{-\delta \, x}\, \Pr\{T_{|h|} \in dx\} \, q_j e^{-q_j \, t} \, dt 
%&=& c_\pm \, \int_{x=0}^\infty  e^{-\delta \, x} 
%    \int_{t=x}^\infty  q_j e^{-q_j \, t} \, dt 	\, \Pr\{T_{|h|} \in dx\} \\
%&=& c_\pm \, \int_{x=0}^\infty  e^{-(\delta+q_j) \, x} 
%     \, \Pr\{T_{|h|} \in dx\} \\
\, = \, c_\pm \, \E_0[e^{-(\delta+q_j) \, T_{|h|}}] \\
&=&  c_\pm \,
{\cosh}\left(\frac{2\mu}{\sigma} \frac{h}{2}\right)
{\rm sech}\left(\frac{2\mu}{\sigma} h \sqrt{\frac{(\delta+q_j)}{2\mu^2} +\frac{1}{4}}\right)
\end{eqnarray*}

Last term in (\ref{eq:V(z,j)}) is positive, it follows that 
\begin{eqnarray*}
V(z,j)
&\leq& g_+(h) \, V(z+h,j) + g_-(h) \, V(z-h,j) \\
%&\leq& \frac{g_+(h)}{g_+(h)+g_-(h)} V(z+h,j) + \frac{g_-(h)}{g_+(h)+g_-(h)} V(z-h,j) \\
&\leq& G_+(h) \, V(z+h,j) + G_-(h) \, V(z-h,j)
\end{eqnarray*}
where $G_\pm(h) = g_\pm(h)/(g_+(h)+g_-(h))$.
Since $G_+(h) + G_-(h) =1$ 
and $G_-(h) \leq \frac{1}{2} \leq G_+(h)$ the follow inequality holds
$$V(z,j)-V(z+h,j) \leq V(z-h,j)-V(z,j)$$
that implies that $V(z,j)$ is continuous and concave in $(0,b(j))$, see 
\citet[Ch. IV.2 page. 326]{courant:1988}. 

Rearranging terms in (\ref{eq:V(z,j)}) and using that $1 - (g_+(h) + g_-(h)) = o(h)$ 
we have
$$
g_-(h) [V(z,j) - V(z-h,j)] = g_+(h) [V(z+h,j) - V(z,j)] + o(h)  \ ,
$$
and dividing it by $h$, letting $h \to 0$ and having $g_\pm(h) \to 1/2$ as $h\to0$
we finally get 
$$V'(z-,j)=V'(z+,j),$$
i.e. the function $V(z,j)$ is differentiable in $(0,b_j)$.
Define
$$\lim_{h \to 0} \frac{V(z+h,j) - 2 V(z,j) + V(z-h,j)}{h^2} = \psi(z)$$
with 
\begin{eqnarray*}
&\psi(z,j) 
= \lim\limits_{h\to0} & \Big[
  \frac{2 g_+(h) - 1}{h} \frac{V(z+h,j)-V(z,j)}{h} \\
&& + \frac{1 - 2 g_-(h)}{h} \frac{V(z,j)-V(z-h,j)}{h}  
   + \frac{2 g_-(h) + 2 g_-(h)-2}{h^2} V(z,j) \\
&& + \frac{h^2 \sum_k q_{jk} \, [V(z \wedge b_k,k)+(z-b_k)^+] + o(h^2)}{h^2} \Big] \ ,
\end{eqnarray*}
where we used the fact that $\tilde V_J(h,k) \to V(z \wedge b_k,k)+(z-b_k)^+$ as $h\to0$ by  bounded
convergence.
Having that $(2 g_\pm(h) - 1)/h \to \pm \mu/\sigma^2$ 
and $(2 g_-(h) + 2 g_-(h)-2)/h^2 \to 2 \left(\delta + q_j\right)/\sigma^2$
as $h\to0$  we get that
\begin{eqnarray}\label{eq:sys.v.second}
 \psi(z,j) 
%&=& \frac{\mu}{\sigma^2} [V'(z+,j) + V'(z-,j) ]
%- \frac{2 \left(\delta +q_j\right)}{\sigma ^2} V(z,j) 
%+ \sum_{k \not = j} q_{jk} \, [V(z \wedge b_k, k)+(z-b_k)^+] \nonumber \\
&=& \frac{2\mu}{\sigma^2} V'(z,j) 
- \frac{2 \left(\delta +q_j\right)}{\sigma ^2} V(z,j)
+ \sum_{k \not = j} q_{jk} \, [V(z \wedge b_k,k)+(z-b_k)^+] \ .
\end{eqnarray}
Using Schwarz's Theorem, \cite{titchmarsh:1976}, we get that 
$V(z,j)$ is twice differentiable and $V''(z,j)=\psi(z,j)$. In addition (\ref{eq:sys.v.second})
coincides with (\ref{eq:sys.v}).

\subsection{Boundary conditions}
To determine the value of $V(z,j)$ is necessary to add to the differential equations 
(\ref{eq:sys.v})
two boundary conditions for each $j\in\Jspace$, at the barriers $0$ and $b(j)$.

It is obvious that $V(0,j)=0$, for any $j\in\Jspace$ and in addition we have
$$ V'(b(j),j)=1. $$
This equation that can be found for the non-modulated case in \cite{gerber:shiu:2004} has the
following
explanation.

Assume $\Delta z$ is small, starting at $b(j)-\Delta z$ we will touch the barrier $b(j)$
at time $T_{\Delta z}$ so that $e^{-\delta \, T_{\Delta z}} = 1 + o(T_{\Delta z})$.
Defining $\tau = T_J \wedge T_0$, the shortest time between a change of state for $J$ or the ruin
epoch, we have that the process $Z(t)$ will have the same dynamic of a single reflected 
$(\mu(j),\sigma(j))-$Brownian motion at the barrier $b(j)$ 
for $0\leq t \leq \tau$ that we denote by $Y(t)$.
The upper regulator process is known to have the following expression
$$U_z(t) = z + \sup\{0\leq s \leq t: Y(t) \vee b(j) - z\}$$
as  $0\leq t \leq \tau$ when $Z(0)=z$, notice that by definition $Y(0)=0$.
Since $\int_0^t e^{-\delta s} \, U(ds) = U(t) + o(t)$, it follows that
$$
V(b(j),j) - V(b(j)-\Delta z, j) 
%&=& \Delta z + \E[\sup\{0\leq s \leq t: Y(t) \} -  \sup\{0\leq s \leq t: Y(t) \vee \Delta z\}] 
%	    + o(\E[T_{\Delta z}])\\
= \Delta z + \E[1\{M(t) < \Delta z\} (M(t) - \Delta z)] + o(\E[T_{\Delta z}]),
$$
where $M(t) = \sup\{0\leq s \leq t: Y(s) \}$ and $T_{\Delta z} = \inf\{t>0, M(t) \geq \Delta z\}$.
Having that 
$$-\Delta z \, \Pr(T_{\Delta z} < \tau)
\leq \E[1\{M(t) < \Delta z\} (M(t)- \Delta z)] \leq 0 $$
we can get the following bounds
$$
1 - \Pr(T_{\Delta z} > \tau) + \frac{o(\E[T_{\Delta z}])}{\Delta z}
\leq \frac{V(b(j),j) - V(b(j)-\Delta z, j)}{\Delta z} \leq 1 + \frac{o(\E[T_{\Delta z}])}{\Delta z},
$$
and taking the limit for $\Delta z \to 0$ and using the fact that 
$\Pr(T_{\Delta z} > \tau) \to 0$ and
$o(\E[T_{\Delta z}])/\Delta z \to 0$
we obtain the result.
%%%%%%%%%%%%%%%%%%%%%%%%%%%%%%%%%%
%%%%%%%%%%%%%%%%%%%%%%%%%%%%%%%%%%
    \subsection{Expected dividend payout - Case $|E|=2$}\label{div.payout.ex}
Assuming $q_{12}=q_{21}=\lambda$, we have that for $z \in I_1$ 
the system of differential equations (\ref{eq:sys.v}) reduces to 
$$
\left\{\begin{array}{ll}
 \frac{1}{2} \sigma^2(1) V''(z,1) + \mu(1) \, V'_1(z,1) 
    - (\lambda+\delta) V(z,1) + \lambda \,  V(z,2) =& 0 \\ 
    \\		
 \frac{1}{2} \sigma^2(2) V''(z,2) + \mu(2) \, V'_1(z,2) 
    + \lambda \,  V(z,1) - (\lambda+\delta) V(z,2)  =& 0 
\end{array} \right.
$$
and for $z \in I_2$ it reduces to 
$$
\hspace{1.1cm} \frac{1}{2} \sigma^2(2) V''(z,2) + \mu(2) \, V'(z,2) 
- (\lambda+\delta) V(z,2) = 
- \lambda (z - b(1) + V(b(1),1)) \ ,
$$
with boundary conditions
\begin{eqnarray}
V(0,1) &=& V(0,2) = 0 \\
V'(b(1),1) &=& V'(b(2),2) = 1
\end{eqnarray}
and regularity conditions
\begin{eqnarray}
 V(b(1+),2)  &=& V(b(1+),2) \\
 V'(b(1+),2) &=& V'(b(1+),2) \ .
\end{eqnarray}

Let $\matF(z)=( \matX_1 \, , \, \matX_2) \, e^{\matJ \, z} \, ( \matX_1^{-1} \, , \,
-\matX_2^{-1})^\t$ be the matrix solution of the system of differential equation
$$
 \frac{1}{2} \, \matDelta^2_{\sigma} \, \matF''(z) + \matDelta_{\mu} \, \matF'(z) +
   \left( \begin{array}{rr}
    - (\lambda+\delta) &  \lambda \\ \lambda &  - (\lambda+\delta) 
  \end{array} \right) \ \matF(z) = \matO
$$
for $z \in I_1$, with the null condition at the 0 barrier, i.e. $\matF(0)= \matO$ and
defining by $z_1 \leq z_2 \leq z_3 \leq z_4$ the four solution of the algebraic equation
%%%%%%%%%%%%%%%%%%%%%%%%%%%%%%%%%%%%%%%%%%%%%%%%%%%%%%%%%%%%%%%%%%%
% CASE WHEN SIGMA^2(1)/2  =  SIGMA^2(2)/2  =  1
% $$
%   z^4 
% + z^3 (\mu(1)+\mu(2)) 
% + z^2 (\mu(1)\,\mu(2)+2(\delta-\lambda))
% + z   (\mu(1)+\mu(2))(\delta-\lambda)
% +     \delta(\delta-2\lambda)
% = 0
% $$
% the matrix $\matJ=\diag[z_1,z_2,z_3,z_4]$ and the matrices $\matX_1$ and $\matX_2$ are given by
% \begin{eqnarray*}
% \matX_1 &=& \left(\begin{array}{cc}
%                  -\frac{\lambda}{(z_1-z_2)(z_1-z_3)(z_1-z_4)} &
% 			-\frac{\lambda}{(z_2-z_1)(z_2-z_3)(z_2-z_4)} \\ \\
% 		 -\frac{\lambda - \delta - z_1 (\mu(1) + z_1)}
% 		      {(z_1-z_2)(z_1-z_3)(z_1-z_4)} &
% 			-\frac{\lambda - \delta - z_2 (\mu(1) + z_2)}
% 		      {(z_2-z_1)(z_2-z_3)(z_2-z_4)}
%                 \end{array}\right) \\
% \matX_2 &=& \left(\begin{array}{cc}
%                  -\frac{\lambda}{(z_3-z_1)(z_3-z_2)(z_3-z_4)} &
% 			-\frac{\lambda}{(z_4-z_1)(z_4-z_2)(z_4-z_3)} \\ \\
% 		 -\frac{\lambda - \delta - z_3(z_3+\mu(1))}
% 		      {(z_3-z_1)(z_3-z_2)(z_3-z_4)} &
% 			-\frac{\lambda - \delta - z_4(z_4 + \mu(1))}
% 		      {(z_4-z_1)(z_4-z_2)(z_4-z_3)}
%                 \end{array}\right)
% \ .
% \end{eqnarray*}
%%%%%%%%%%%%%%%%%%%%%%%%%%%%%%%%%%%%%%%%%%%%%%%%%%%%%%%%%%%%%%%%%%%
% GENERAL CASE
\begin{eqnarray*}
  z^4 
&+& z^3 \left[\frac{2\mu(1)}{\sigma^2(1)} + \frac{2\mu(2)}{\sigma^2(2)}\right] \\
&+& z^2 \left[\frac{2\mu(1)}{\sigma^2(1)} \, \frac{2\mu(2)}{\sigma^2(2)} 
	+ (\delta-\lambda) \left(\frac{2}{\sigma^2(1)}+\frac{2}{\sigma^2(2)}\right)\right]\\
&+& z   \left[(\delta-\lambda) \frac{4(\mu(1)+\mu(2))}{\sigma^2(1) \, \sigma^2(2)}\right]
 = \frac{ 4 \delta (2\lambda-\delta)}{\sigma^2(1) \, \sigma^2(2)}
\end{eqnarray*}
the matrix $\matJ=\diag[z_1,z_2,z_3,z_4]$ and the matrices $\matX_1$ and $\matX_2$ are given by
\begin{eqnarray*}
\matX_1 &=& \left(\begin{array}{cc}
                 -\frac{2\lambda/\sigma^2(2)}{(z_1-z_2)(z_1-z_3)(z_1-z_4)} &
			-\frac{2\lambda/\sigma^2(2)}{(z_2-z_1)(z_2-z_3)(z_2-z_4)} \\ \\
		 -\frac{2 (\lambda-\delta)/\sigma^2(1) - z_1 (2 \mu(1)/\sigma^2(1) + z_1 )}
		      {(z_1-z_2)(z_1-z_3)(z_1-z_4)} &
			-\frac{2 (\lambda-\delta)/\sigma^2(1) - z_2 (2 \mu(1)/\sigma^2(1) + z_2)}
		      {(z_2-z_1)(z_2-z_3)(z_2-z_4)}
                \end{array}\right) \\
\matX_2 &=& \left(\begin{array}{cc}
                 -\frac{2\lambda/\sigma^2(2)}{(z_3-z_1)(z_3-z_2)(z_3-z_4)} &
			-\frac{2\lambda/\sigma^2(2)}{(z_4-z_1)(z_4-z_2)(z_4-z_3)} \\ \\
		 -\frac{2 (\lambda-\delta)/\sigma^2(1) - z_3 (2 \mu(1)/\sigma^2(1) + z_3 )}
		      {(z_3-z_1)(z_3-z_2)(z_3-z_4)} &
			-\frac{2 (\lambda-\delta)/\sigma^2(1) - z_4 (2 \mu(1)/\sigma^2(1) + z_4 )}
		      {(z_4-z_1)(z_4-z_2)(z_4-z_3)}
                \end{array}\right)
\ .
\end{eqnarray*}
The above expressions are a rearrangement of formulas obtained by using the software 
Mathematica\copyright.
We have that
$$
\begin{array}{llr}
\vec V(z) &= \matF(z) \, (k_1, \, k_2)^\t 
	&0 \leq z \leq b(1) \\ \\
V(z,2) &= k_3 \, f(z) + g(z) - \frac{1}{\lambda+\delta} V(b(1),1) 
	&b(1) < z \leq b(2) 
\end{array}
$$
where 
$$g(z)= \frac{\lambda}{\lambda+\delta} z
        + \frac{\lambda \, \mu(2)}{(\lambda+\delta)^2}
        - \frac{\lambda}{\lambda+\delta} b(1)
$$ 
and $f(z)$ being any solution of the differential system 
$$
\frac{1}{2} \sigma^2(2) f''(z) 
+ \mu(2) \, f'(z,2) 
- (\lambda+\delta) f(z) = 0 \ ,
$$
with $b(1) < z \leq b(2)$ and boundary condition $f'(b(2)) = \delta/(\lambda+\delta)$. 
We choose the special solution that has
$$f(b(2)) = - \frac{1}{\Theta} \, f'(b(2)) 
	=  -\frac{\delta}{\lambda+\delta} \frac{\sigma^2(2)}{\mu(2)} $$
that can be written as
$$f(z) = -\frac{\delta}{\lambda+\delta} \frac{\sigma^2(2)}{\mu(2)}
	\, \exp\Big(\Theta (b(2) -  z) \Big) \,
	 \cosh \Big(\Delta (b(2) - z)\Big) \ ,$$
with $\Theta = \frac{\mu(2)}{\sigma^2(2)}$ and 
$\Delta = \sqrt{\Theta^2 + \frac{2(\lambda+\delta)}{\sigma^2(2)}}$.

To solve for the constants, $k_1$, $k_2$ and $k_3$ we need to solve the following
system of equations
\begin{eqnarray*}
\vec V'(b(1)-) &=& ( 1 \, , \, V'(b(1)+,2) )^\t \\
V(b(1)-,2) &=& V(b(1)+,2)
\end{eqnarray*}
that is equal to 
\begin{eqnarray}
\matF'(b(1)) \, (k_1 \, , \, k_2)^\t 
&=& (1 \, , \, k_3 \, f'(b(1)) + g'(b(1)) )^\t \label{eq:deriv}\\
e_2^\t \, \matF(b(1)) \, (k_1 \, , \, k_2)^\t &=& k_3 \, f(b(1)) + g(b(1))
 - \frac{1}{\lambda+\delta} e_1^\t \, \matF(b(1)) \, (k_1 \, , \, k_2)^\t   \label{eq:cont}
\end{eqnarray}

Using (\ref{eq:deriv}) we get
\begin{eqnarray*}
 (k_1 \, , \, k_2)^\t 
%&=&
%\matF'(b(1))^{-1} \, (1 \, , \, k_3 \, f'(b(1)) + g'(b(1)) )^\t \\
&=&
\matF'(b(1))^{-1} \, (1 \, , \, k_3 \, f'(b(1)) + \lambda/(\lambda+\delta))^\t
\end{eqnarray*}
and using the (\ref{eq:cont}) we get \iffalse
\begin{eqnarray*}
k_3 \, f(b(1))
&=&
\left( 1/(\lambda+\delta) \, , \, 1 \right) \, \matF(b(1)) \, (k_1 \, , \, k_2)^\t
- g(b(1)) \\
&=&
\left( 1/(\lambda+\delta) \, , \, 1 \right) \, \matF(b(1)) \, 
\matF'(b(1))^{-1} \, (1 \, , \, k_3 \, f'(b(1)) + g'(b(1)) )^\t
- g(b(1)) \\
&=&
\left( 1/(\lambda+\delta) \, , \, 1 \right) \, \matF(b(1)) \, 
\matF'(b(1))^{-1} \, (1 \, , \, g'(b(1)) )^\t - g(b(1)) \\
&\phantom{=}& +
\left( 1/(\lambda+\delta) \, , \, 1 \right) \, \matF(b(1)) \, 
\matF'(b(1))^{-1} \, e_2 \, k_3 \, f'(b(1)) \ .
\end{eqnarray*}
Therefore \fi
\begin{eqnarray*}
k_3
%&=& - \frac{%
%g(b(1)) - \left( 1/(\lambda+\delta) \, , \, 1 \right) \, \matF(b(1)) \, 
%\matF'(b(1))^{-1} \, (1 \, , \, g'(b(1)) )^\t %
%}
%{%
% f(b(1)) - f'(b(1)) %
% \left( 1/(\lambda+\delta) \, , \, 1 \right) \, \matF(b(1)) \, 
%\matF'(b(1))^{-1} \, e_2 \ .
%}\\
&=& - \frac{%
\lambda \, \mu(2)
 - \left( 1 \, , \, \lambda+\delta \right) \, \matF(b(1)) \, 
\matF'(b(1))^{-1} \, (\lambda+\delta \, , \, \lambda)^\t %
}
{%
 (\lambda+\delta)^2 \, f(b(1)) - (\lambda+\delta) \, f'(b(1)) %
 \left( 1 \, , \, \lambda+\delta \right) \, \matF(b(1)) \, 
\matF'(b(1))^{-1} \, e_2 \ .
}
\end{eqnarray*}
%%%%%%%%%%%%%%%%%%%%%%%%%%%%%%%%%%
%%%%%%%%%%%%%%%%%%%%%%%%%%%%%%%%%%
%    \input{DividendPayout_2States_Opt}
%%%%%%%%%%%%%%%%%%%%%%%%%%%%%%%%%%
%%%%%%%%%%%%%%%%%%%%%%%%%%%%%%%%%%
%%%%%%%%%%%%%%%%%%%%%%%%%%%%%%%%%%
%%%%%%%%%%%%%%%%%%%%%%%%%%%%%%%%%%
%************************************************************************
%		APPENDIX				
%************************************************************************
% \appendix
% \section{Know results}
% For sake of completeness we list in this appendix some known result.
% \begin{theorem}[{\citet[Lemma V.21.13]{rogers:williams:2000}}]\label{thm:rogers.williams}
%  Let $g$ be a function
%  $$g:(0,\infty)\times E \times E \times \Omega \to \R$$
%  with the following properties:
%  \begin{enumerate}
%   \item for all $i$, $j$ the map $(t,\omega) \to g(t,i,j,\omega)$ is locally bounded previsible;
%   \item for all $t$, $\omega$ and $i$, $g(t,i,i,\omega)=0$
%  \end{enumerate}
%  Then
%     $$N^g_t(\omega) \defeq \sum_{0< s \leq t} g(s,J(s^-),J(s),\omega)-\int_0^t \sum_k q_{J(s^-),k}
%     \, g(s,J(s^-),k,\omega)ds$$
%     defines a local martingale $N^g$.
% \end{theorem}

% ----------------------------------------------------------------
%Included for Gather Purpose only:
%input "../../Bibliography/bibfile.bib"
%\bibliography{../../Bibliography/jabRef}
%\bibliographystyle{elsarticle-harv}
%\bibliographystyle{../Bibliography/mystyle}
%\bibliography{../Bibliography/JabRef}

% ----------------------------------------------------------------
\end{document}